\documentclass[a4, 12pt]{amsart}
\usepackage{amsmath,enumerate}
\usepackage{amsmath,  amscd, amsfonts, amssymb, graphicx, color, comment, ulem, amsthm}

\usepackage[plainpages=false,colorlinks,pdfpagemode=None,bookmarksopen,linkcolor=blue,citecolor=blue,urlcolor=blue]{hyperref}
\usepackage{cite}

\theoremstyle{proclaim}
\newtheorem{theorem}{Theorem}[section]
\newtheorem{lemma}[theorem]{Lemma}
\newtheorem{corollary}[theorem]{Corollary}
\newtheorem{problem}[theorem]{Problem}
\newtheorem{proposition}[theorem]{Proposition}
\theoremstyle{fancyproclaim}

\theoremstyle{statement}
\newtheorem{remark}[theorem]{Remark}

\newtheorem{example}[theorem]{Example}
\theoremstyle{fancystatement}

\numberwithin{equation}{section}

\providecommand{\AMS}{$\mathcal{A}$\kern-.1667em%
\newcommand{\R}{\mathbb{R}}
\newcommand{\N}{\mathbb{N}}
\newcommand{\eps}{\varepsilon}
\newcommand{\supp}{\mathrm{supp}}
\numberwithin{equation}{section} \linespread{1.2}
\lower.25em\hbox{$\mathcal{M}$}\kern-.125em$\mathcal{S}$}

\newcommand{\mm}{\mathfrak{m}}
\newcommand{\App}{\mathcal{A}_{+}^{-1}}
\newcommand{\A}{\mathcal{A}}
\newcommand{\B}{\mathcal{B}}
\newcommand{\Bpp}{\mathcal{B}_{+}^{-1}}
\newcommand{\Ap}{\mathcal{A}_+}
\newcommand{\Bp}{\mathcal{B}_+}
\newcommand{\bhp}{B(H)_+}
\newcommand{\M}{\mathcal{M}}
\newcommand{\E}{\tilde{E}}
\newcommand{\PSupp}{\operatorname{psupp}}
\newcommand{\Supp}{\operatorname{supp}}

\begin{document}
\date{April 13, 2021; to appear in Journal of Operator Theory}

\setcounter{page}{1}

\title[Transformations preserving the norm of means]{Transformations preserving the norm of means between positive cones of general and commutative $C^*$-algebras}

\author[Dong]{Yunbai Dong}

\address[Dong]{School of Mathematics and Computer, Wuhan Textile University, Wuhan 430073, China.}

\email{baiyunmu301@126.com}

\author[Li]{Lei Li}
\address[Li]{School of Mathematical Sciences and LPMC, Nankai University,
Tianjin 300071, China.}

\email{leilee@nankai.edu.cn}

\author[Moln\'{a}r]{Lajos Moln\'{a}r}
\address[Moln\'{a}r]{Bolyai Institute, University of Szeged, H-6720 Aradi v\'{e}rtan\'{u}k tere1, Szeged, Hungary and
Budapest University of Technology and Economics, Institute of Mathematics, H-1521 Budapest, Hungary}
\email{molnarl@math.u-szeged.hu}

\author[Wong]{Ngai-Ching Wong}
\address[Wong]{Department of Applied Mathematics, National Sun Yat-sen University, Kaohsiung, Taiwan, \and School of Mathematical Sciences, Tiangong University, Tianjin 300387, China.}
\email{wong@math.nsysu.edu.tw}



\begin{abstract}
In this paper, we consider a (nonlinear) transformation $\Phi$ of invertible positive elements
in $C^*$-algebras which preserves the norm of any of the three fundamental means of positive elements; namely,
$\|\Phi(A)\mm \Phi(B)\| = \|A\mm B\|$, where $\mm$ stands for
the arithmetic mean $A\nabla B=(A+B)/2$,
the geometric mean $A\#B=A^{1/2}(A^{-1/2}BA^{-1/2})^{1/2}A^{1/2}$,
or the harmonic mean $A!B=2(A^{-1} + B^{-1})^{-1}$.
Assuming that $\Phi$ is surjective and preserves either the norm of the arithmetic mean or the norm of the geometric mean, we show that $\Phi$ extends to a Jordan $*$-isomorphism between the underlying full algebras.
If $\Phi$ is surjective and preserves the norm of the harmonic mean, then we obtain the same
conclusion in the   special cases where
the underlying algebras are $AW^*$-algebras or commutative $C^*$-algebras.

In the commutative case,  for a transformation $T: F(\mathrm{X})\subset C_0(\mathrm{X})_+\rightarrow C_0(\mathrm{Y})_+$,
we can relax the surjectivity assumption and show that $T$ is a generalized composition operator
if $T$ preserves the norm of the (arithmetic, geometric, harmonic, or in general any power) mean of any  finite collection of positive functions, provided that
the  domain $F(\mathrm{X})$ contains  sufficiently many elements to peak on compact $G_\delta$ sets.
When the image $T(F(\mathrm{X}))$ also contains sufficiently many elements
 to peak on compact $G_\delta$ sets,  $T$ extends to an algebra $*$-isomorphism
between the underlying full function algebras.
\end{abstract}
\maketitle

\section{Introduction}

The celebrated Mazur-Ulam theorem \cite{MU32}  states that
if $T$ is a surjective map  between   normed spaces $E$ and $F$ preserving the norm of differences, i.e.,
$
\|Tx-Ty\|=\|x-y\|,\enskip x,y\in E,
$
then  $T$ is automatically affine.
Being aware of this classical result,
one may ask what happens if $T$ preserves the norm of sums instead of differences, i.e., if
$
\|Tx+Ty\|=\|x+y\|, \enskip x,y\in E.
$
Such  transformations are called \textit{norm-additive} maps and were studied, e.g., in \cite{HosseiniFont,Gaal,chen, ton}.
In the above setting, this turns out to be an easy problem.
Indeed, inserting $-x$ into the place of $y$ above,
we have $T(-x)=-Tx$, and then immediately arrive at the original Mazur-Ulam theorem.

Observe that the previous argument does not work if the domain of $T$ is not the whole linear space.
This leads to the following open problem proposed in \cite{zhang}.

\begin{problem}\label{pr101}
Suppose that $E$, $F$ are ordered Banach spaces with positive cones $E_+,$ $F_+$, respectively. Let $T: E_+\rightarrow F_+$ be a surjective norm-additive map. Can $T$ be extended to a positive real-linear isometry from $E$ onto $F$?
\end{problem}

Affirmative answers to Problem \ref{pr101} are given for  $C(\mathrm{X})$ spaces in  \cite{chen}, and
for $L_p$ spaces when $p\in\, (1,\infty]$, with a counterexample   for the case $p=1$, in \cite{zhang}.
It also has a positive answer for von Neumann algebras in \cite{ML20b} and for noncommutative $L_p$ spaces whenever $p\in (1,\infty)$ in \cite{ZTW-nonLp}.
Bijective maps between positive definite cones of unital $C^*$-algebras which are norm-additive
with respect to a sort of Schatten $p$-norm for some $p\in\, (1,\infty)$ were described in \cite{Gaal}.

A norm-additive map can obviously be viewed as a transformation which preserves the norm of the arithmetic mean of pairs of elements. If,
on the positive definite cone $\App$ of a unital $C^*$-algebra $\A$,
we replace the arithmetic mean by the geometric mean and the $C^*$-norm by the trace norm
(corresponding to a faithful tracial positive linear functional), then we have got a quantity which is a
variant of quantum R\'enyi relative entropy.  Hence, the maps which preserve that quantity are kinds of
quantum symmetries, and the description of them is presented in  \cite{ML19c}.

These provide us   motivations to study  surjective maps between positive  cones of
$C^*$-algebras which preserve the norm of any of the three most fundamental means, namely,
\begin{itemize}
  \item the \textit{arithmetic mean} $A\nabla B=(A+B)/2$,
  \item the \textit{geometric mean} $A\#B=A^{1/2}(A^{-1/2}BA^{-1/2})^{1/2}A^{1/2}$, and
  \item the \textit{harmonic mean} $A!B=2(A^{-1} + B^{-1})^{-1}$.
\end{itemize}
(See \cite{KubAnd80} for details and other important Kubo-Ando means.)
In this general setting, we provide the following affirmative answer to Problem \ref{pr101},
which consists of Theorems \ref{T:2}, \ref{thm:GM} and \ref{thm:gen-hm} in
Section \ref{S:2}, for the important class of abstract $C^*$-algebras.
Note that all such surjective maps of positive elements preserve the norm, and positive  surjective linear
isometries between $C^*$-algebras are exactly Jordan $*$-isomor\-phisms due to Kadison's theorem \cite{Kad}.

\begin{theorem}\label{T:3all}
Let $\A,\B$ be unital $C^*$-algebras, and let
$\phi:\App\to \Bpp$ be a surjective map (not necessarily additive or positive homogeneous) between their positive definite cones.
Then $\phi$ satisfies
\begin{equation*}\label{E:200}
\|\phi(A)\,\mm\,\phi(B)\|=\|A\,\mm\, B\|, \quad A,B\in \App,
\end{equation*}
where $\mm$ is the arithmetic mean or the geometric mean,
if and only if there is a Jordan $*$-isomorphism $J:\A\to \B$ which extends $\phi$, i.e., $\phi(A)=J(A)$ holds  for all $A\in \App$.
A similar assertion holds for the harmonic mean in the special case when $\A,\B$ are both $AW^*$-algebras or they are
both commutative $C^*$-algebras.
\end{theorem}

Without the surjectivity assumption above, even a norm-additive map $\phi: C_0(\mathrm{X})_+\to C_0(\mathrm{Y})_+$ between positive cones of abelian $C^*$-algebras
can carry no linearity structure.  To see this, we first recall
the following locally compact version 
of Holszty\'{n}ski's nonsurjective Banach-Stone theorem \cite{hol}.

\begin{proposition}[{Jeang and Wong \cite{JW96}}]\label{th102}
Let $\mathrm{X}, \mathrm{Y}$ be locally compact Hausdorff spaces. Suppose $T: C_0(\mathrm{X})\rightarrow C_0(\mathrm{Y})$ is a (not necessarily surjective) linear isometry. Then there is a locally compact subset $\mathrm{Y}_0$ of $\mathrm{Y}$, a surjective continuous function $\tau$ from $\mathrm{Y}_0$ onto $\mathrm{X}$, and a continuous
unimodular scalar valued function $g$ on $\mathrm{Y}_0$ such that
\begin{align}\label{eq:wco}
(Tf)(y)=g(y)\cdot f(\tau (y)), \quad y\in \mathrm{Y}_0, f\in C_0(\mathrm{X}).
\end{align}
\end{proposition}

We call maps of the form \eqref{eq:wco} \textit{generalized weighted composition operators} and, in the cases where $g\equiv 1$, \textit{generalized composition operators}.
The word ``generalized'' refers to the fact that, in general, $\mathrm{Y}_0$ is only a subset of $\mathrm{Y}$.
In fact, although $\mathrm{Y}_0$ is big enough such that $\tau(\mathrm{Y}_0)=\mathrm{X}$,
outside the `good part' $\mathrm{Y}_0$ of $\mathrm{Y}$, the behavior of the linear isometry $T$
is totally out of control (except that $|Tf(y)|\leq \|f\|$ holds for all $y\in \mathrm{Y}\setminus \mathrm{Y}_0$).
If $\mathrm{Y}_0$ happens to be the whole space $\mathrm{Y}$, then we call the above maps \textit{weighted composition operators} and \textit{composition operators}, respectively.

In view of Proposition \ref{th102}, we might expect that a nonsurjective norm-additive map between positive cones of
$C_0(\mathrm{X})$-spaces can be extended to a generalized composition operator.
Indeed, Tonev and Yates \cite{ton} consider  such maps between uniform algebras and give sufficient conditions to
ensure them being composition operators.   See also, among others, \cite{Sady1,HosseiniFont,Sady2}.
However, we have the following example of a norm-additive map between positive cones of continuous functions
which is not a generalized composition operator.

\begin{example}\label{eg:2notn}
Let $\mathrm{X}=\{x_0\}$ be a single point space, and identify the cone $C(\mathrm{X})_{+}^{-1}$ of all
strictly positive functions on $\mathrm{X}$ with the interval $(0,\infty)$.
Let
$$
A_1=\, (0,1],\quad A_2=\, (1,2], \quad\text{and}\quad A_3=\, (2,\infty).
$$
Define $\varphi_i:\, (0,\infty)\,\longrightarrow\, (0,\infty)$ by
\begin{align}\varphi_i(t) =\left\{\begin{array}{ll}
t/2,&\,\,\,\,  t\in A_i,\\
t,&\,\,\,\, t\notin A_i,\end{array} \quad i=1,2,3.
\right. \nonumber\end{align}
We construct a map $T:\, (0,\infty)\,\longrightarrow C([0,1])_{+}^{-1}$ by
 \begin{align}(Tt)(y)=\left\{\begin{array}{ll}
(1-2y)\varphi_1(t)+2y\varphi_2(t),&\quad  y\in[0,1/2],\\
(2-2y)\varphi_2(t)+(2y-1)\varphi_3(t),&\quad y\in \, (1/2,1],\end{array}\right.\nonumber
\end{align}
for all $t\in \, (0, \infty)$.
Then one can easily check that $T$ is norm-additive but there is not any point
$y_0\in [0,1]$ such that $(Tt)(y_0)=t$ holds for all $t\in \, (0,\infty)$. Hence $T$ is not a generalized composition operator as in
\eqref{eq:wco}; note that $\tau(y)=x_0$, and because $T$ preserves the norm,  one must have there $g(y)=1$ for all $y\in Y_0$.
\end{example}

In Section \ref{S:3},
 we offer a way to fix this issue. Namely, we define
\textit{finitely norm-additive} maps, that is, those transformations $T$ which satisfy
$$\|T x_1+T x_2+\cdots +T x_n\|=\|x_1+x_2+\cdots +x_n\|$$ for all elements $x_1,x_2,\ldots ,x_n$ of their domains and for any $n\in \mathbb{N}$.
Apparently, this is the same as to assume that $T$ preserves the norm of the arithmetic mean of any finite collection of elements,
 that is,
\[
\left\|\frac{T x_1+T x_2+\cdots + T x_n}{n}\right\|=\left\|\frac{x_1+x_2+\cdots +x_n}{n}\right\|,\quad  x_1,x_2,\ldots x_n,\,  n\in \mathbb{N}.
\]
We will see in Theorem \ref{th301} that such maps are necessarily generalized composition operators.
Referring back to Example \ref{eg:2notn}, observe that the map $T$ defined there,
though being norm-additive, it is not finitely norm-additive. Indeed, one can check that
$(T1+T2+T3)(y)={11/2}-y$ for all $y\in [0,1]$, and hence $\|T1+T2+T3\|={11/2}\neq 6=1+2+3=\|T_1\|+\|T_2\|+\|T_3\|$.

With the success in describing the maps preserving the norm of the arithmetic mean in Theorem  \ref{th301}, we also study
maps preserving the norm of the geometric mean, or
the norm of the harmonic mean, of any finite collections of positive functions.
We arrive at similar conclusions for these two important means in Theorems \ref{thG01} and \ref{thH01}, respectively.

One can define other means of finitely many positive functions.
Classical examples include the power means $\mm^p$ defined by
$$
f_1 \,\mm^p\, f_2 \,\mm^p\, \cdots \mm^p f_n = \left(\frac 1n(f_1^p + f_2^p +\cdots + f_n^p)\right)^{1/p}.
$$
When $p <0$, we should consider strictly positive functions on
compact Hausdorff spaces only.  The power means cover the arithmetic mean with $p=1$ and the harmonic
mean with $p=-1$.  By convention we set $\mm^0$ to be the geometric mean, and we note that $\lim_{p\to 0} \mm^p=\mm^0$.

Summarizing and extending Theorems \ref{th301}, \ref{thG01} and  \ref{thH01},
the following result covers the cases of all power means,
 especially the arithmetic, geometric and harmonic means.
Here by saying that  $F(\mathrm{X})\subset C_{0}(\mathrm{X})_+$
\textit{contains sufficiently many functions to peak on   compact $G_\delta$ subsets of $\mathrm{X}$}, we
mean that for any nonempty compact $G_\delta$ subset $K$ of $\mathrm{X}$, there is an $f$ in $F(\mathrm{X})$ such that
$K=\{x\in \mathrm{X}  : f(x)=\|f\|\}$.

\begin{theorem}\label{thm:p-means}
When $0\leq p< \infty$ (resp.\ $-\infty < p<0$), let $\mathrm{X},\mathrm{Y}$ be locally compact (resp.\ compact) Hausdorff spaces,
and let $F(\mathrm{X})$ be a subset of $C_0(\mathrm{X})_{+}$
(resp.\  $C_0(\mathrm{X})_{+}^{-1}$)
containing sufficiently many functions to peak on  compact $G_\delta$ subsets of $\mathrm{X}$.
Let  $T: F(\mathrm{X}) \rightarrow C_{0}(\mathrm{Y})_{+}$ (resp.\  $T: F(\mathrm{X}) \rightarrow  C_{0}(\mathrm{Y})_{+}^{-1}$)
be a (not necessarily additive, positive homogeneous or surjective) map preserving the norm of the power mean $\mm^p$
of any finite collection of elements in $F(\mathrm{X})$; namely,
$$
\|((Tf_1)^p +\cdots + (Tf_n)^p))^{1/p}\| = \|((f_1^p +\cdots + f_n^p)^{1/p}\|,
$$
for all $f_1, \ldots, f_n\in F(\mathrm{X}), n=1,2, \ldots$. Then there exist a locally compact
subset $\mathrm{Y}_0$ of $\mathrm{Y}$ and a surjective continuous map $\tau: \mathrm{Y}_0\rightarrow \mathrm{X}$ such that
\begin{align*}
(Tf)(y)= f(\tau (y)), \quad y\in \mathrm{Y}_0, f\in F(\mathrm{X}).
\end{align*}
In particular, $T$ is injective. In the case when $\mathrm{X}$ is compact, $\mathrm{Y}_0$ is a closed subset of $\mathrm{Y}$.

If the range $T(F(\mathrm{X}))$ of $T$ also contains sufficiently many functions to peak on  compact $G_\delta$ subsets of $\mathrm{Y}$, then we have $\mathrm{Y}_0=\mathrm{Y}$ and $\tau$ is a homeomorphism from $\mathrm{Y}$ onto $\mathrm{X}$.
\end{theorem}

The following Corollary \ref{cor:3means-sphere}
is a consequence of Theorem \ref{thm:p-means}.
This may be an interesting result  to those  working on
the recently popular Tingley's problem (see, e.g., \cite{P18AM, P18TAMS, P18JMAA, P18}).  For example,  the Open Problem 5 in \cite{P18}
asks if a surjective map between the positive unit spheres
of two $C^*$-algebras preserving the norm of differences can be extended to a surjective isometry between the whole spaces.
Here, we consider the norm of mean preserver version of the problem for the case of commutative $C^*$-algebras.

Let $\mathrm{X}$ be a locally compact Hausdorff space.
Let $C_{0}(\mathrm{X})_{+,1}$  be the positive unit sphere consisting of
norm one positive continuous functions on $\mathrm{X}$.
When $\mathrm{X}$ is compact, let
$C(\mathrm{X})_{+,1}^{-1}=C_{0}(\mathrm{X})_{+,1}\cap C_{0}(\mathrm{X})^{-1}_{+}$
be the positive definite unit sphere of $C(\mathrm{X})$.

\begin{corollary}\label{cor:3means-sphere}
Let $\mathrm{X}$, $\mathrm{Y}$ be locally compact (resp.\ compact) Hausdorff spaces.
Assume that  $T:C_{0}(\mathrm{X})_{+,1}\rightarrow C_{0}(\mathrm{Y})_{+,1}$
(resp.\  $T:C(\mathrm{X})^{-1}_{+,1}\rightarrow C(\mathrm{Y})^{-1}_{+,1}$)
is a map
preserving the norm of the power mean $\mm^p$ with $p\geq 0$, especially the
arithmetic or the geometric mean (resp.\ with $p<0$, especially the harmonic mean)  of arbitrary finitely many elements.
Then there exist a locally compact subset
$\mathrm{Y}_0$ of $\mathrm{Y}$, which is closed if $\mathrm{X}$ is compact,
and a surjective continuous map $\tau: \mathrm{Y}_0\rightarrow \mathrm{X}$ such that
\begin{align}\label{eg:Tingley-rep}
(Tf)(y)= f(\tau (y)), \quad y\in \mathrm{Y}_0,\ f\in C_{0}(\mathrm{X})_{+,1}\ \ (\text{resp.}\ f\in C(\mathrm{X})^{-1}_{+,1}).
\end{align}
If  $T$ has a dense range, then $\mathrm{Y}_0=\mathrm{Y}$ and $\tau$ is a homeomorphism from $\mathrm{Y}$ onto $\mathrm{X}$, and thus $T$ extends to an algebra $*$-isomorphism between the full function algebras.
\end{corollary}

We will present the proofs of Theorem \ref{thm:p-means} and Corollary \ref{cor:3means-sphere}  in the subsection \ref{Ss:3.4}.

Finally, we remark that although there is also a structure theorem
for nonsurjective linear isometries between general $C^*$-algebras in \cite{CW04},
it is far more complicated than the one given in  Proposition \ref{th102}.
Meanwhile,  similar questions for nonsurjective maps preserving the norm of means between general $C^*$-algebras do not seem traceable.

\section{Surjective maps preserving the norm of means between positive   cones of $C^*$-algebras}\label{S:2}

The set of all positive elements of a
$C^*$-algebra $\A$, that is the set of all self-adjoint elements with nonnegative spectrum, is denoted by $\Ap$, and we call it the
\textit{positive (semidefinite) cone} of $\A$. For self-adjoint elements $A,B\in \A$ we write $A\leq B$ if and only if $B-A\in \Ap$.
When $\A$ is unital, the set $\App$ of all invertible elements in $\Ap$ is called the \textit{positive definite cone} of $\A$.

Let $\A,\B$ be  $C^*$-algebras.  A bijective linear map $J:\A \to \B$ is called a
Jordan $*$-isomorphism if
$J(A^*)=J(A)^*$ and   $J(A^2)=J(A)^2$, $A\in \A$; the second condition is equivalent to that
\[
J(AB+BA)=J(A)J(B)+J(B)J(A), \quad A,B\in \A.
\]
Jordan $*$-isomorphisms are the fundamental symmetries of $C^*$-algebras.
For various characterizations of them we refer to \cite{ML13j}. In particular,
it follows from Kadison's extension \cite{Kad} of the classical Banach-Stone theorem  that the positive
surjective linear isometries between $C^*$-algebras are exactly the Jordan $*$-isomorphisms.
We also have the following special case of  Theorem 9 in \cite{ML14a} (also see Theorem 16 in \cite{ML19c}).

\begin{proposition}[Hatori and Moln\'{a}r \cite{ML14a}, Moln\'{a}r \cite{ML19c}]\label{prop-Monlar}
Let $\phi: \App\to \Bpp$ be a positive homogeneous order isomorphism between the positive definite cones of  unital $C^*$-algebras.
In other words,
  \begin{enumerate}[(1)]
    \item $\phi$ is a bijective map,
    \item $A\leq B$ if and only if $\phi(A)\leq \phi(B)$, and
    \item $\phi(tA)=t\phi(A)$
  \end{enumerate}
holds for every $A,B\in \App$ and positive real number $t$.
  Then there is a $C\in \Bpp$ and a Jordan $*$-isomorphism $J: \A\to \B$ such that $\phi(A)= CJ(A)C$ for all $A\in \App$.
\end{proposition}

Lemma \ref{lem:key} below provides us another
interesting nonlinear characterization of Jordan $*$-isomorphisms on the positive cones of $C^*$-algebras.
For its proof, we need

\begin{lemma}\label{L:3}
Let $\A$ be a unital $C^*$-algebra and $A,B\in \App$. The following assertions are equivalent:
\begin{itemize}
\item[(i)] $A=B$;
\item[(ii)] for any $C\in \Ap$ and positive real number $\delta$, we have
\begin{align*}
\{X\in \Ap\, :\, CX=0,\, \|X\|\leq \delta\}\ &=\ \{X\in \Ap\, :\, CX=0,\, \|X\|\leq \delta,\, X\leq A\}
\intertext{if and only if}
\{X\in \Ap\, :\, CX=0,\, \|X\|\leq \delta\}\ &=\ \{X\in \Ap\, :\, CX=0,\, \|X\|\leq \delta,\, X\leq B\}.
\end{align*}
\end{itemize}
\end{lemma}

\begin{proof}
The implication (i) $\Rightarrow$ (ii) is obvious. Assume now that (ii) holds.
Assume also that, on the contrary, $T=A^{-1}-B^{-1}$ is nonzero and has, for example, a negative element in its spectrum.
Let $\lambda_0$ be a negative real number such that for the spectral measure $E_T$ of $T$ we have $E_T(\!-\infty, \lambda_0)\neq 0$.
Observe that
\begin{eqnarray*}
\lambda_0 E_T(\!-\infty, \lambda_0)&\geq& E_T(\!-\infty, \lambda_0)TE_T(\!-\infty, \lambda_0)\\
&=& E_T(\!-\infty, \lambda_0)(A^{-1}-B^{-1})E_T(\!-\infty, \lambda_0).
\end{eqnarray*}
This implies that
\begin{eqnarray*}
E_T(\!-\infty, \lambda_0)B^{-1}E_T(\!-\infty, \lambda_0)\geq E_T(\!-\infty, \lambda_0)A^{-1}E_T(\!-\infty, \lambda_0)-\lambda_0 E_T(\!-\infty, \lambda_0).
\end{eqnarray*}
Taking norms, we have
\begin{equation}\label{E:10}
\| E_T(\!-\infty, \lambda_0)B^{-1} E_T(\!-\infty, \lambda_0)\| \geq \| E_T(\!-\infty, \lambda_0)A^{-1} E_T(\!-\infty, \lambda_0)\|+|\lambda_0|.
\end{equation}

For a real number $\epsilon>0$,
pick a continuous real function $f_\epsilon$ with values in $[0,1]$ such that $f_\epsilon(t)=0$ for $t\leq \lambda_0-\epsilon$ and $f_\epsilon(t)>0$ for $t> \lambda_0 -\epsilon$. Let $\delta>0$. We assert
$$
\{X\in \Ap\,:\, f_\epsilon(T)X=0, \| X\|\leq \delta\}=
\{X\in \Ap\,:\, X\leq \delta E_T(\!-\infty ,\lambda_0-\epsilon]\}.
$$
In fact, only the inclusion of the former set in the latter one needs proof. If $f_\epsilon(T)X=0$, then
$Xf_\epsilon^{1/n}(T)=0$ for every positive integer $n$ which then yields that $XE_T(\lambda_0-\epsilon, \infty)=0$.
Therefore, if $X\in \Ap$, $f_\epsilon(T)X=0$, $\|X\|\leq \delta$, we necessarily have
$X\leq \delta E_T(\!-\infty, \lambda_0-\epsilon]$.

We assert that for a particular positive number $\delta$ satisfying that
$$
1/\delta \geq \|E_T(\!-\infty, \lambda_0)A^{-1}E_T(\!-\infty, \lambda_0)\|,
$$
we have
\begin{equation}\label{La:1}
\begin{gathered}
\{X\in \Ap\, :\, X \leq \delta E_T(\!-\infty, \lambda_0-\epsilon]\}=
\{X\in \Ap\, :\, f_\epsilon(T)X=0, \| X\| \leq \delta \}\\
=\{X\in \Ap\, :\, f_\epsilon(T)X=0, \| X\| \leq \delta, X\leq A \}.
\end{gathered}
\end{equation}
The first equality is already known.
To verify the second one, we show that
$$
\delta E_T(\!-\infty, \lambda_0)\leq A.
$$
Indeed, this inequality holds if and only if
$$
A^{-1/2}E_T(\!-\infty, \lambda_0)A^{-1/2}\leq (1/\delta)I
$$
which amounts to saying that
$$
\|A^{-1/2}E_T(\!-\infty, \lambda_0)A^{-1/2}\|\leq 1/\delta,
$$
or equivalently,
$$
\|E_T(\!-\infty, \lambda_0)A^{-1}E_T(\!-\infty, \lambda_0)\|\leq 1/\delta.
$$
The last condition is true by our choice of $\delta$. Using the first equality in \eqref{La:1}, we then have the second equality there.

From \eqref{La:1} together with the assumption (ii), it follows that for any chosen number
$\delta$ such that $1/\delta \geq \|E_T(\!-\infty, \lambda_0)A^{-1}E_T(\!-\infty, \lambda_0)\|$ and  for any $X\in \Ap$, we have the implication
$$
X\leq \delta  E_T(\!-\infty, \lambda_0-\epsilon] \quad\implies\quad X\leq B.
$$
Since there is a monotone increasing sequence of nonnegative continuous functions
which converges pointwise to the indicator function of the open interval $(\!-\infty, \lambda_0-\epsilon)$,
we obtain that
$$
\delta  E_T(\!-\infty, \lambda_0-\epsilon)\leq B, \quad \epsilon>0,
$$
which
means
$$
B^{-1/2}E_T(\!-\infty, \lambda_0)B^{-1/2}\leq 1/\delta.
$$
This gives us that
$$
\| E_T(\!-\infty, \lambda_0)B^{-1} E_T(\!-\infty, \lambda_0)\|\leq 1/\delta.
$$
Since this holds for any positive  $\delta$ such that
$1/\delta \geq\|E_T(\!-\infty, \lambda_0)A^{-1}E_T(\!-\infty, \lambda_0)\|$, we
obtain a contradiction to \eqref{E:10}. This completes the proof of the lemma.
\end{proof}

\begin{lemma}\label{lem:key}
Let $\A$ and $\B$ be $C^*$-algebras such that at least one of them is unital.
  Let $\phi:\A_+\to \B_+$ be a surjective map between the positive cones of $\A$ and $\B$.
  Suppose that $\phi$  preserves the norm, the order and the orthogonality in two directions; namely,
  the conditions
  \begin{enumerate}[(1)]
  \item $\|\phi(A)\|=\|A\|$,
  \item $A\leq B$ if and only if $\phi(A)\leq \phi(B)$, and
  \item $AB = 0$ if and only if $\phi(A)\phi(B)=0$
  \end{enumerate}
  are satisfied for any $A,B\in \Ap$.
Then $\phi$ extends to a Jordan $*$-isomorphism $J:\A\to \B$.
\end{lemma}
\begin{proof}
First note that $\phi$ is bijective and sends the positive part $\operatorname{B}(\A)_+$ of the closed unit ball of $\A$ onto
the positive part $\operatorname{B}(\B)_+$ of the closed unit ball of  $\B$.
Assume that  $\A$  has identity $I$.  Since $I$ is the greatest element of $\operatorname{B}(\A)_+$,
the image $\phi(I)$ is also the greatest element of $\operatorname{B}(\B)_+$.
It implies that $\B$ is also unital with the identity $I=\phi(I)$.  The case that $\B$ has an identity is similar.
Therefore, we can assume that both $\A,\B$ are unital.

It follows from the norm preserving property that
$\phi(tI)=tI$ holds for every positive real number $t$.
Indeed, we have
$\phi(A)\leq tI$ if and only if $\|\phi(A)\|\leq t$, which is equivalent to $\|A\|\leq t$.
The latter holds if and only if $A \leq tI$,
and this is equivalent to $\phi(A)\leq \phi(tI)$ since $\phi$ is an order isomorphism. We then obtain the desired equality $\phi(tI)=tI$.
Using the order preserving property and this equality, it follows that $\phi$ preserves the invertible positive elements in both directions.

Now, select any $A\in \App$ and positive real number $t$. Then for any $D\in \Bp$ and $\delta>0$ we have
the following equivalences.
\begin{align*}
&\quad \{Y\in \Bp\, :\, DY=0, \|Y\|\leq \delta\}\\
&\qquad=\{Y\in \Bp\, :\, DY=0,\,  \|Y\|\leq \delta,\,  Y\leq \phi(tA)\}\\
\Longleftrightarrow
&\quad \text{for any } Y\in \Bp\ \text{we have}\  DY=0,\, \|Y\|\leq \delta \Rightarrow Y\leq  \phi(tA)\\
\Longleftrightarrow
&\quad \text{for any } X\in \Ap\ \text{we have}\  \phi^{-1}(D)X=0,\,  \|X\|\leq \delta \Rightarrow X\leq tA\\
\Longleftrightarrow
&\quad \text{for any } X\in \Ap\ \text{we have}\  \phi^{-1}(D)(X/t)=0,\,  \|X/t\|\leq \delta/t \Rightarrow X/t\leq A\\
\Longleftrightarrow
&\quad \text{for any } Z\in \Ap\ \text{we have}\  \phi^{-1}(D)Z=0,\,  \|Z\|\leq \delta/t \Rightarrow Z\leq A\\
\Longleftrightarrow
&\quad \text{for any } Y\in \Bp\ \text{we have}\  DY=0,\,  \|Y\|\leq \delta/t \Rightarrow Y\leq \phi(A)\\
\Longleftrightarrow
&\quad \text{for any } Y\in \Bp\ \text{we have}\  D(tY)=0,\,  \|tY\|\leq \delta \Rightarrow tY\leq t\phi(A)\\
\Longleftrightarrow
&\quad \{W\in \Bp\, :\, DW=0, \|W\|\leq \delta\}\\
&\qquad=\{W\in \Bp\, :\, DW=0,\,  \|W\|\leq \delta,\,  W\leq t\phi(A)\}.
\end{align*}
By Lemma \ref{L:3}, this yields $\phi(tA)=t\phi(A)$. It follows that $\phi$, when restricted to $\App$, is a bijection onto $\Bpp$ which preserves the order in both directions and is positive homogeneous.
By Proposition \ref{prop-Monlar},
there exist a Jordan $*$-isomorphism $J:\A\to \B$ and an invertible positive element $C\in \B$ such that $\phi(A)=CJ(A)C$ holds for all $A\in \App$. Since $\phi(tI)=tI$, it follows that $C=I$. This gives us that
$\phi(A)=J(A)$ for all $A\in \App$.

To show that this former equality is valid for any $A\in \Ap$, observe that
$$
\phi(A + tI) = J(A+tI) = J(A) + tI, \quad t>0.
$$
Since $\phi$ is an order isomorphism and $A  = \inf\{A+tI: t >0\}$, we have
\begin{align*}
 \phi(A) = \inf\{\phi(A+tI): t>0\}= \inf \{J(A + tI): t>0\} = J(A), \quad A\in \Ap.
\end{align*}
\end{proof}

\subsection{Preservers of the norm of the arithmetic mean}

The following result describes the precise structure of surjective maps between positive definite cones of $C^*$-algebras,
which preserve the norm of the arithmetic mean.

\begin{theorem}\label{T:2}
Let $\A,\B$ be unital $C^*$-algebras. Assume that $\phi:\App\to \Bpp$ is a surjective map. Then $\phi$ satisfies
\begin{equation}\label{E:2}
\|\phi(A)+\phi(B)\|=\|A+B\|, \quad A,B\in \App,
\end{equation}
if and only if there is a Jordan $*$-isomorphism $J:\A\to \B$ which extends $\phi$, i.e., $\phi(A)=J(A)$ holds  for all $A\in \App$.
\end{theorem}

The proof of Theorem \ref{T:2} will be easy after we verify the corresponding result for such maps between positive semidefinite cones.

\begin{theorem}\label{T:1}
Let $\A,\B$ be $C^*$-algebras such that at least one of them is unital.
Assume that $\phi:\Ap\to \Bp$ is a surjective map. Then $\phi$ satisfies
\begin{equation}\label{E:1}
\|\phi(A)+\phi(B)\|=\|A+B\|, \quad A,B\in \Ap,
\end{equation}
if and only if there is a Jordan $*$-isomorphism $J:\A\to \B$ which extends $\phi$.
\end{theorem}

Theorem \ref{T:1} gives an affirmative answer for Problem \ref{pr101} concerning the general class of unital $C^*$-algebras. We recall that in the particular case of von Neumann algebras, the result was proved in Theorem 2.7 in \cite{ML20b}.
Not surprisingly, the problem for general $C^*$-algebras is significantly more difficult to resolve.
By Kadison's result (in fact, by the isometric property of Jordan $*$-isomor\-phisms), the sufficiency parts of the former two statements are obvious.
For the necessity parts, we apply Lemma \ref{lem:key}.

It is plain that  a surjective map $\phi$ satisfying \eqref{E:1} preserves the norm; namely,
$$
\|\phi(A)\|=\frac{\|\phi(A) + \phi(A)\|}{2} = \frac{\|A+A\|}{2}=\|A\|, \quad A\in \Ap.
$$
The following shows that such a $\phi$ is necessarily an order isomorphism, and thus $\phi$ is bijective.

\begin{lemma}\label{L:1}
Let $\mathcal A$ be a $C^*$-algebra. For any $A,B\in \Ap$ we have
$A\leq B$ if and only if $\|A+X\| \leq \|B+X\|$ holds for all $X\in \Ap$
(or for all $X\in \App$ if $\A$ is unital).
\end{lemma}

\begin{proof}
The nontrivial part of the statement is the sufficiency.

Assume first that $\A$ is unital and
that $\|A+X\| \leq \|B+X\|$ holds for all $X\in \App$. Pick a positive scalar $t$ such that $X=tI-B\in \App$.
We have $\|A+(tI-B)\|\leq t$ which implies that $A+(tI-B)\leq tI$, and thus $A\leq B$.

Suppose now $\A$ does not have an identity. Apparently, without  loss of generality, we can assume that $\|A+B\|= 1$ and
we also can assume that $\A$ is a norm closed *-subalgebra of the full operator algebra over a Hilbert space.  Observe that
$B\leq A+B\leq (A+B)^{1/n}\uparrow P=E_{A+B}(0,\infty)$, which is
the support projection of $A+B$, as $n\to\infty$ in the strong operator topology.  It follows from
$$
\|A + ((A+B)^{1/n}-B)\| \leq \|B  + (((A+B)^{1/n}-B)\| = \|A+B\|^{1/n}=1
$$
that
$
A + P - B \leq I.
$
Since $E_A(0,\infty) = \sup_n A^{1/n} \leq \sup_n (A+B)^{1/n} = P$ and similarly
$E_B(0,\infty)\leq P$, we have
$$
A=PAP\leq P(I-P +B)P =B.
$$
\end{proof}

To prove that the map in Theorem \ref{T:1} preserves orthogonality in both directions,
we need the following characterization of zero product of positive elements.

\begin{lemma}\label{L:2}
Let $\mathcal A$ be a $C^*$-algebra. For any $A,B\in \Ap$, the following assertions are equivalent
\begin{enumerate}[(1)]
\item $AB=0$,
\item $\|C+D\| =\max\{\|C\|,\|D\|\}
\ \text{holds for any}\  C,D\in \Ap\ \text{with}\ C\leq A \ \text{and}\ D\leq B$.
\end{enumerate}
\end{lemma}

\begin{proof}
The implication from (1) to (2) is easy. As for the converse, we can assume that $\A$ is a norm closed *-subalgebra of the full operator algebra
over a Hilbert space $H$. Pick $A,B\in \Ap$ such that
for any $C,D\in \Ap$ with $C\leq A$, $D\leq B$ we have $\|C+D\|=\max\{\|C\|,\|D\|\}.$ Let $E_A,E_B$ be the spectral measures of $A,B$, respectively.
Consider the indicator function of an open interval $(a,\infty)$, where $a>0$. Clearly, there is a monotone increasing sequence $\{g_n\}$ of
nonnegative continuous functions  vanishing at zero, which converges pointwise to that indicator function.
It follows that $ag_n(A), ag_n(B)\in \Ap$ and  $ag_n(A)\leq A$, $ag_n(B)\leq B$ for every $n\in \mathbb N$. Moreover, we have that
$g_n(A) \to E_A(a,\infty)$, $g_n(B) \to E_B(a,\infty)$ monotone increasingly in the strong operator topology.
Hence $ag_n(A)+ag_n(B)\to aE_A(a,\infty)+aE_B(a,\infty)$ monotone increasingly in the strong operator topology. It follows that
$\|ag_n(A)+ag_n(B)\|\to \|aE_A(a,\infty)+aE_B(a,\infty)\|$.
On the other hand, by our assumption,
\begin{eqnarray*}
\|ag_n(A)+ag_n(B)\|&=&\max\{\|ag_n(A)\|,\|ag_n(B)\|\}\\
&\to& \max\{\|aE_A(a,\infty)\|, \|aE_B(a,\infty)\|\}.
\end{eqnarray*}
It then follows that
$$\|E_A(a,\infty)+E_B(a,\infty)\|=\max\{\|E_A(a,\infty)\|, \|E_B(a,\infty)\|\}.$$
If one of the projections $E_A(a,\infty), E_B(a,\infty)$ is nonzero, then we have
$$\|E_A(a,\infty)+E_B(a,\infty)\|= 1$$
implying that
$E_A(a,\infty)+E_B(a,\infty)\leq I$.
Consequently, the projections $E_A(a,\infty)$, $E_B(a,\infty)$ are orthogonal to each other. Since it holds for all $a>0$, we conclude
\[E_A(0,\infty) E_B(0,\infty)=0.\]
This means that the projections onto the closures of the ranges of $A$ and $B$ are orthogonal and hence, we have $AB=0$.
\end{proof}

We are now in a position to prove Theorem \ref{T:1}.

\begin{proof}[Proof of Theorem \ref{T:1}]
Assume that $\phi:\Ap\to \Bp$ is a surjective map which satisfies \eqref{E:1}. It is plain that $\phi$ preserves the norm.
By Lemma \ref{L:1}, $\phi$ preserves the order in both
directions.
By Lemma \ref{L:2}, $\phi$ preserves orthogonality in both directions.
It then follows from Lemma \ref{lem:key} that there is a Jordan $*$-isomorphism $J:\A\to \B$ extending $\phi$.
\end{proof}

For the proof of Theorem \ref{T:2} we need the following simple lemma.

\begin{lemma}\label{lem:App2Ap}
Let $\A,\B$ be unital $C^*$-algebras.
  Let $\phi: \App\to \Bpp$ be an order isomorphism fixing all positive scalar multiples of the identity.
   For any $\epsilon >0$, define $\psi:\Ap\to \Bp$ by $\psi(A)=\phi(A+\epsilon I) - \epsilon I$.
  Then $\psi$ is an order isomorphism from $\Ap$ onto $\Bp$.  Moreover, if $\phi$ preserves the norm, then so does $\psi$.
\end{lemma}
\begin{proof}
  For any given positive number $\epsilon$,  the map $\psi:\Ap\to \Bp$  defined by
$\psi(A)=\phi(\epsilon I+A)-\epsilon I$ is clearly  injective.
To see that $\psi$ is a surjective map, we observe that for any $B\in \Bp$, there is an $A_\epsilon\in \App$ such that
$\phi(A_\epsilon) = B + \epsilon I$.  Since $\phi^{-1}$ is an order isomorphism sending $\epsilon I$ to $\epsilon I$,
we see that
\[A_\epsilon = \phi^{-1}(B + \epsilon I) \geq \phi^{-1}(\epsilon I) = \epsilon I.\]
Therefore, $A= A_\epsilon - \epsilon I \in \A_+$ exists and gives $\psi(A)=B$.
It is plain that $\psi$ is  an order isomorphism.

We show that $\psi$  preserves the norm when $\phi$ does.
Observe that for $B\in \A_+$ with $B \geq \epsilon I$ we have $\|B - \epsilon I\| = \|B\| -\epsilon$.  Therefore, for any $A\in \A_+$ we can compute
\begin{align*}
\|\psi(A)\| &= \|\phi(A + \epsilon I) - \epsilon I\| = \|\phi(A + \epsilon I)\| - \epsilon\\
&= \|A+\epsilon I\| - \epsilon = \|(A+\epsilon I) - \epsilon I\|=\|A\|.
\end{align*}
\end{proof}

\begin{proof}[Proof of Theorem \ref{T:2}]
Assume that $\phi:\App\to \Bpp$ is a surjective map satisfying \eqref{E:2}.
We can show just as in the proof of Theorem \ref{T:1}
that $\phi$ is an order isomorphism (in particular $\phi$ is injective) and, following the argument in the second paragraph of the proof of Lemma \ref{lem:key} that $\phi$
fixes all positive scalar multiples of the identity.
By Lemma \ref{lem:App2Ap}, we define an order isomorphism $\psi: \Ap\to \Bp$ by
setting
$\psi(A)=\phi(\epsilon I+A)-\epsilon I$ for any $A\in \Ap$.
We also have
\begin{eqnarray*}
\|\psi(A)+\psi(B)\|&=&\|\phi(\epsilon I +A)+\phi(\epsilon I +B)-2\epsilon I\|\\
&=&\|\phi(\epsilon I +A)+\phi(\epsilon I +B)\|-2\epsilon\\
&=&\|(\epsilon I +A)+(\epsilon I +B)\|-2\epsilon \\
&=&
\| A+B\|
\end{eqnarray*}
for any $A, B\in \Ap$.

Applying Theorem \ref{T:1}, we obtain that $\psi$ extends to a Jordan $*$-isomorphism between $\A$ and $\B$.
In particular, $\psi$ is affine. It now easily follows that $\phi$ is also affine on the set of all elements of $\App$
which are bounded from below by $\epsilon I$. Since this holds for any positive $\epsilon$,
we obtain that $\phi$ is affine. The structure of affine bijections between positive definite cones of unital $C^*$-algebras
is known; see, e.g.,
 Proposition 1 in \cite{ML17e}. That result says that those maps can be obtained as the composition
 $A\mapsto CJ(A)C$ of a Jordan $*$-isomorphism $J$ between the
 underlying $C^*$-algebras and a two-sided multiplication by an invertible positive element $C$.
Since we have $\phi(I)=I$, that invertible positive element $C$ is the identity and we conclude that $\phi$ extends to
the  Jordan $*$-isomorphism $J$. This completes the proof of the theorem.
\end{proof}

\subsection{Preservers of the norm of the geometric mean}

We next consider the Kubo-Ando geometric mean $\#$ and the preservers of its norm. On the positive definite cone $\App$ of
a unital $C^*$-algebra $\A$, the operation $\#$ is defined as follows
\[
A\#B=A^{1/2}(A^{-1/2}BA^{-1/2})^{1/2}A^{1/2}, \quad A,B\in \App.
\]
As a part of Theorem 1 in \cite{CMM} we have the following statement.

\begin{theorem}[Chabbabi,  Mbekhta and  Moln\'ar \cite{CMM}]\label{thm:GM}
Let $\A, \B$ be unital $C^*$-algebras and assume that $\phi: \App \to \Bpp$ is a surjective map. It satisfies
\begin{equation*}\label{E:11}
\|\phi(A)\# \phi(B)\|= \|A\#B\|, \quad A,B\in \App
\end{equation*}
if and only if there is a Jordan $*$-isomorphism $J:\A\to \B$ extending $\phi$.
\end{theorem}

In fact, the result in \cite{CMM} was formulated also for weighted geometric means. Its proof was based on a characterization of the order similar to Lemma \ref{L:1}. In Lemma 9 in \cite{CMM} it is shown that for any $A, B \in \App$ we have
$$
A\leq B\quad\text{if and only if}\quad \|A\# X\|\leq \|B\# X\|, \quad X\in \App
$$
(the proof  of this statement was far more difficult than that of Lemma \ref{L:1} here).

\subsection{Preservers of the norm of the harmonic mean}

As the third most important type of means, the harmonic mean $A!B$ is defined on the positive definite cone $\App$
of a unital $C^*$-algebra $\A$ by
\[
A!B=2(A^{-1} + B^{-1})^{-1}, \quad A,B\in \App.
\]
Let $\A,\B$ be unital $C^*$-algebras and
$\phi:\App\to \Bpp$ be a surjective map which satisfies
\begin{equation}\label{E:5}
\|\phi(A)!\phi(B)\|=\|A!B\|, \quad A,B\in \App.
\end{equation}
In order to describe $\phi$, we
observe that, analogous to  Lemma \ref{L:1}, we have

\begin{lemma}\label{lem:HM-order}
Let $\A$ be a unital $C^*$-algebra.
For any $A,B\in \App$ we have
$$
A\leq B\quad\text{if and only if}\quad
\|A!X\| \leq \| B!X\|, \quad X\in \App.
$$
\end{lemma}
\begin{proof}
We verify the sufficiency only.
Let $X=(tB-I)^{-1}B$ for any $t> 0$ with $tB- I\in \App$.
One can verify that the following implications are valid
$$
\|A!X\| \leq \|B!X\| = {2}/{t} \quad\implies \quad t \leq {A}^{-1} + (tB-I)B^{-1} \quad\implies\quad
A\leq B.
$$
\end{proof}

It follows from Lemma \ref{lem:HM-order} that the map $\phi$ in \eqref{E:5} above
is necessarily an order isomorphism between $\App$ and $\Bpp$.
Clearly, it also preserves the norm of elements.
Therefore,
as in the proof of Lemma \ref{lem:key},
we obtain that
 $\phi(tI)=tI$ holds for every positive real number $t$.

Unfortunately, the above observation is practically all what we can say.
We do not have any further idea from here in the description of $\phi$ in the setting of general $C^*$-algebras.
Namely, we are unable to prove that $\phi$ is positive homogeneous which would be necessary to apply
Proposition \ref{prop-Monlar} to get the structure of $\phi$ as we have done for the case of the arithmetic mean.

However, something can still be done for less general $C^*$-algebras. The reason is that there are situations where we do not need to assume positive homogeneity of order
isomorphisms and we are still able to get the complete description of those isomorphisms.
The first of such results were given in \cite{Sch79} for order isomorphisms of positive cones of commutative
$C^*$-algebras.  Although the description of such maps can be very complicated in general (see Theorem 6.1 in \cite{Sch79}),
the one for those fixing positive scalars just fits our task; namely, we have

\begin{proposition}[Sch\"{a}ffer {\cite[Corollary 8.4]{Sch79}}]\label{prop:Schaffer}
Let $\mathrm{X},\mathrm{Y}$ be compact Hausdorff spaces.  Let $\phi: C(\mathrm{X})_+^{-1}\to C(\mathrm{Y})_+^{-1}$ be an order isomorphism such that
$\phi(tg) = t\phi(g)$ for
some $g\in C(\mathrm{X})_+^{-1}$ and all $t>0$.  Then there is a homeomorphism $\sigma: \mathrm{Y}\to \mathrm{X}$ and an $h\in C(\mathrm{Y})_+^{-1}$ such that
  $$
  \phi(f)(y) = h(y)f(\sigma(y)), \quad f\in C(\mathrm{X})_+^{-1},\, y\in \mathrm{Y}.
  $$
\end{proposition}
In particular, if $\phi(t\mathrm{1}_\mathrm{X})=t\mathrm{1}_\mathrm{Y}$ for $t>0$ in above,
where $\mathrm{1}_\mathrm{X}, \mathrm{1}_\mathrm{Y}$ are the constant one functions
on $\mathrm{X}, \mathrm{Y}$, respectively, then $h=\mathrm{1}_\mathrm{Y}$ and $\phi(f)=f\circ\sigma$
for $f\in C(\mathrm{X})_+^{-1}$. Hence
$\phi$ extends to an algebra $*$-isomorphism between commutative unital $C^*$-algebras.
It thus follows from Proposition \ref{prop:Schaffer} the special case of Theorem \ref{T:3all}
for the preservers of the norm of the harmonic mean  in commutative $C^*$-algebras (to see that those preservers fix the constant functions, we refer to the second paragraph in the proof of Lemma \ref{lem:key}).

For the noncommutative case, we have also useful results from
\cite{ML01f} and \cite{ML11a} in relation with the algebra $B(H)$ of all bounded linear operators on a Hilbert space $H$.
It was proved in Theorem 1 in \cite{ML01f} that, if $\dim H\geq 2$, every order isomorphism of $\bhp$ is of the form $A\mapsto CAC^*$ with some invertible bounded linear or conjugate linear operator $C$ on $H$.
In Theorem 1 in \cite{ML11a} the same conclusion was deduced for the order isomorphisms of $\bhp^{-1}$.
Recently, in \cite{Mori}, Mori has substantially extended those results (and also the results of \v Semrl in \cite{Semrl} which describe all order isomorphisms of general operator intervals in $B(H)$) to the setting of von Neumann algebras.

\begin{proposition}[Mori {\cite[Theorem 4.3]{Mori}}]\label{prop:Mori}
Let $\mathcal{M}$ be a von Neumann algebra without commutative direct summand.
Let $\phi:\mathcal{M}_{+}^{-1} \to \mathcal{N}_{+}^{-1}$ be an order isomorphism.
Then
\[\phi(A)=CJ(A)C^*, \quad A\in \mathcal{M}_{+}^{-1},\]
where $J:\mathcal{M}\to \mathcal{N}$ is a Jordan $*$-isomorphism and $C\in \mathcal{N}$ is an
invertible element.
\end{proposition}

Since for our map $\phi$ satisfying \eqref{E:5} we have $\phi(tI)=tI$, $t>0$ (again, see the second paragraph in the proof of Lemma \ref{lem:key}), it follows  that $C$ here is necessarily
a unitary. Therefore, we obtain another special case of Theorem \ref{T:3all} for preservers of the norm of the harmonic mean
in the setting of von Neumann algebras without commutative direct summand.
However, with Proposition \ref{prop:Schaffer}
we also have a positive result for the commutative case.
Theorem \ref{thm:gen-hm} below covers the case of all $AW^*$-algebras and hence that of all general von Neumann algebras.

For the formulation and the proof of Theorem \ref{thm:gen-hm} we recall the following.
A $C^*$-algebra $\mathcal{M}$ is said to be an \textit{$AW^*$-algebra}
if  any set of orthogonal projections in $\mathcal{M}$ has a least upper bound, and
that each maximal commutative  $C^*$-subalgebra of $\mathcal{M}$ is generated by its
projections.  It is equivalent to that
for every nonempty set $S\subset\mathcal{M}$ there is a projection
$P\in  \mathcal{M}$ such that
$$
S^\bot:=\{X\in \mathcal{M}: XY=0\ \text{holds for all}\ Y\in S\} = \mathcal{M}P
$$
(see, e.g., Exercise 1C on page 43 in \cite{BerBook}).

Clearly, a von Neumann algebra is an $AW^*$-algebra.  For any positive element $X$ in an $AW^*$-algebra $\mathcal{M}$ of Hilbert space operators, the range projection
$E_X(0,\infty)$ of $X$ belongs to $\mathcal{M}$ (see \cite[Theorem 7]{Frank}).
Moreover, for projections $P,Q$ in $\M$ the infimum $P\wedge Q$ exists in $\M$, which is the projection $R$
such that $\{I-P,I-Q\}^{\bot} = \mathcal M R$.
On the other hand,
a commutative $AW^*$-algebra is algebra $*$-isomorphic to $C(\Omega)$ for some Stonean space $\Omega$, i.e., a compact Hausdorff space
in which the closure $\overline{O}$ of
any open subset $O$ is open (see Theorem 1 on page 40 in \cite{BerBook}).

Let $\mathcal{A}$ be a $C^*$-subalgebra of an $AW^*$-algebra $\mathcal{M}$.
Call $\mathcal{A}$ an \textit{$AW^*$-subalgebra} of $\mathcal{M}$ if $\mathcal{A}$ itself is an $AW^*$-algebra, and
the supremum of any family of orthogonal projections in $\mathcal{A}$, computed in $\mathcal{M}$,
belongs to $\mathcal{A}$ (see Proposition 8(i) on page 23 in \cite{BerBook}).
Every maximal commutative $C^*$-subalgebra of $\mathcal{M}$ is an $AW^*$-subalgebra of $\mathcal{M}$
(cf. Proposition 8(iv) on page 23 in \cite{BerBook}).
The intersection of a family of $AW^*$-subalgebras of $\mathcal{M}$ is again an $AW^*$-subalgebra of $\mathcal{M}$
(see Proposition 8(ii) on page 23 in \cite{BerBook}).
For any nonempty subset $S$ of $\mathcal{M}$, the $AW^*$-subalgebra $AW^*(S)$ generated by $S$ is the intersection of
all $AW^*$-subalgebras of $\mathcal{M}$ containing $S$.
In the case when $S$ is a selfadjoint and commuting family of elements in $\mathcal{M}$,
the $AW^*$-subalgebra $AW^*(S)$ is  commutative,
as $S$ is contained in a maximal commutative $C^*$-subalgebra of $\mathcal{M}$.

The $AW^*$-subalgebra  $AW^*(I,X)$  of $\M$ generated by the
identity $I$ and a positive element $X$  is a commutative $AW^*$-subalgebra of $\mathcal{M}$.
Consequently, we can think of $AW^*(I,X)$ as $C({\Omega})$ for a Stonean space $C({\Omega})$ with
$f$ in $C({\Omega})$ corresponding to $X$ and the constant 1 function corresponding to  $I$.
We  associate to $X$ a `spectral family' $\{\E_X(t): t\geq0\}$ of projections contained in $\M$.
Namely, these projections are computed in  $C({\Omega})$  such that
$\E_X(t)$ is the indicator function of the clopen set $\overline{f^{-1}(t, \infty)}$.
It is easy to see that $\E_X(t)X=X\E_X(t)$ for all $t\geq 0$. Moreover, if $\mathcal M$ is an $AW^*$-algebra of Hilbert space operators, then it can be shown that
$$
E_X(t,\infty) \leq \E_X(t) \leq E_X[t,\infty), \quad t\geq 0.
$$
Here, $E_X(t,\infty)$ and $E_X[t,\infty)$ are the spectral projections
of $X$ which might not
belong to $\mathcal{M}$ while $\E_X(t)\in \mathcal{M}$.
Moreover, we have
$$
t\E_X(t)\leq X\E_X(t)\leq X \quad  \text{and}\quad X(I-\E_X(t)) \leq t(I-\E_X(t)), \quad t\geq 0.
$$

In the next lemma we present a characterization of the orthogonality of positive elements.

\begin{lemma}\label{lem:AW-ortho}
For any positive elements $X,Y$ in an $AW^*$-algebra $\mathcal{M}$, we have
$XY=0$ if and only if  $\E_X(t)\E_Y(t)=0$ holds for all  $t >0$.
\end{lemma}
\begin{proof}
To see the `only if' part, we can consider the commutative $AW^*$-algebra $C(\mathrm{X})$ generated by $I, X$ and $Y$, in which
the supports of the continuous functions corresponding to $X$ and $Y$ are disjoint clopen subsets of the Stonean space $\mathrm{X}$.

For the `if' part,
observe that
$$XY=(X-X \E_X(\epsilon))Y+ X \E_X(\epsilon)(Y-\E_Y(\epsilon)Y)+X \E_X(\epsilon)\E_Y(\epsilon)Y.
$$
From this we derive
$$
\|XY\| \leq \|X-X \E_X(\epsilon)\|\,\|Y\| + \|X\|\,\|Y- \E_Y(\epsilon)Y\|  \leq \epsilon(\|X\|+\|Y\|).
$$
As $\epsilon >0$ can be arbitrary small, we see that $XY=0$.
\end{proof}

Let $X,Y$ be positive elements in a $C^*$-algebra $\A$. By $X\wedge Y=0$ we mean that there is no nonzero positive $Z$ in $\A$ such
that $Z\leq X$ and $Z\leq Y$.  It is apparent that in this case we have $(tX)\wedge Y=0$ for any positive $t>0$.
For a nonempty subset $F\subset \Ap$, let
$$
F^\wedge = \{Y\in \Ap: X\wedge Y = 0\ \text{for all}\ X\in F\},
$$
and $F^{\wedge\wedge}= (F^\wedge)^\wedge$.
Note that
$X\in F^\wedge$ implies $tX\in F^\wedge$ for all $t>0$.
Note also that if $\psi: \Ap\to \Bp$ is an order isomorphism between positive cones of $C^*$-algebras
then $\psi(F^\wedge)=\psi(F)^\wedge$ holds for any subset $F\subset \Ap$.

\begin{lemma}\label{lem:HM-ww}
Let $\mathcal{M}$ be an $AW^*$-algebra, and let $X\in \mathcal{M}_+$.  Then
\begin{eqnarray*}
&& \{Y\in \mathcal{M}_+ : Y\leq \|Y\|\E_X(t)\ \text{for some}\ t>0\}\\
 &\subset& \{X\}^{\wedge\wedge}
  \subset \{Y\in \mathcal{M}_+ : \E_Y(t)\leq \E_X(0)\ \text{for all} \ t>0\} .
  \end{eqnarray*}
\end{lemma}
\begin{proof}
  Let $Y\in \mathcal{M}_+$ such that $Y\leq \|Y\|\E_X(t)$ \text{for some} $t>0$. Let $U\in \mathcal{M}_+$ such that $U\wedge X=0$ but
  $W\leq U$ and $W\leq Y$ for some nonzero $W\in \M_+$.
  Then there is a nonzero projection $P= \E_W(s)\in \mathcal{M}_+$ for some  $s>0$
  such that $sP\leq W\leq U$ and $sP\leq W\leq Y\leq \|Y\|\E_X(t)\leq (1/t)\|Y\|X$.
  Consequently, $sP$ is less than both $U$ and  $(1/t)\|Y\|X$.  While $ U\wedge (t\|Y\|X)=0$,
  we have $P=0$.  This contradiction establishes the first inclusion.

  On the other hand, let $Y\in \{X\}^{\wedge\wedge}$, and so does every  $\E_Y(t)$ for $t>0$.
  Let $Q_t=\E_Y(t)\wedge \E_X(0)$. Then
  $(\E_Y(t)-Q_t)\wedge \E_X(0)=0$.
  We claim that $(\E_Y(t)-Q_t)\wedge X=0$.
  For else, as in the first part, we will have a nonzero projection $P$ and a scalar $s>0$ such that
  $sP\leq \E_Y(t) -Q_t$ and $sP\leq X$.
  Since the map $a\mapsto a^{1/n}$ is operator monotone in $\M_+$, we have $s^{1/n}P\leq \E_Y(t) -Q_t$ and
  $s^{1/n}P\leq X^{1/n}$ for $n=1,2,\ldots$.  As $n\to \infty$, we have
  $P\leq \E_Y(t) -Q_t$ and $P\leq \E_X(0)$.  Hence $P\leq (\E_Y(t) -Q_t)\wedge \E_X(0)=0$.
  This contradiction verifies our claim.  Consequently,
   $\E_Y(t)-Q_t\in \{X\}^\wedge$.  This forces
  $\E_Y(t)\wedge (\E_Y(t)-Q_t)=0$.  In other words, $\E_Y(t) = Q_t\leq \E_X(0)$ for all $t>0$.
\end{proof}

The following lemma gives a characterization of the
orthogonality of positive elements in terms of the norm of harmonic means.

\begin{lemma}\label{lem:hm-ortho}
Let $\mathcal{M}$ be an $AW^*$-algebra and $t >0$ be a given number.
For $X,Y\in \mathcal{M}_{+}$, we have
$XY = 0$ if and only if
$$
\|(tI + U)!(tI + V)\|\leq \dfrac{2t(t+\max\{\|U\|,\|V\|\})}{2t+\max\{\|U\|,\|V\|\}}
$$
holds for any $U\in \{X\}^{\wedge\wedge}$ and $V\in \{Y\}^{\wedge\wedge}$.
\end{lemma}
\begin{proof}
For the `only if' part, suppose $X,Y\in\mathcal{M}_+$ are such that  $XY=0$.
If  $U\in \{X\}^{\wedge\wedge}$ and $V\in \{Y\}^{\wedge\wedge}$, then
by Lemmas \ref{lem:AW-ortho} and \ref{lem:HM-ww} we see that $UV=0$.
Using functional calculus, we can assume that $U,V, I$ are
nonnegative continuous functions on some compact Hausdorff space
with $I$ being the constant 1 function.   Then,
$\|U+V\|= \max\{\|U\|,\|V\|\}$.
Moreover,
\begin{align*}
\|(tI  + U)!(tI  + V)\| &= \left\|\dfrac{2(tI+U)(tI+V)}{2tI + U + V}\right\| \\
&=
\left\|\dfrac{2t(tI+U+V)}{2tI +U + V}\right\|=\left\| 2t - \frac{2t^2}{2tI +U +V}\right\| \\
& = 2t - \frac{2t^2}{2t +\max\{\|U\|,\|V\|\}}= \dfrac{2t(t+\max\{\|U\|,\|V\|\})}{2t+\max\{\|U\|,\|V\|\}}.
\end{align*}

For the `if' part suppose the inequality in the statement holds for nonzero $X,Y$ in $\mathcal{M}_{+}$.
Suppose $s>0$ is such that both the  projections $P=\E_X(s)$ and $Q=\E_Y(s)$   are nonzero.
By Lemma \ref{lem:HM-ww}, $P\in \{X\}^{\wedge\wedge}$ and $Q\in \{Y\}^{\wedge\wedge}$.
By assumption, we have
$$
\|(tI+P)!(tI+Q)\|\leq \dfrac{2t(t+\max\{\|P\|,\|Q\|\})}{2t+\max\{\|P\|,\|Q\|\}} = \dfrac{2t(t+1)}{2t + 1}.
$$
Since
$$
(tI+P)^{-1}= \dfrac{I}{t} - \dfrac{P}{t(1+t)},
$$
we have
\begin{align*}
  (tI+P)!(tI+Q) &= 2\big( (tI+P)^{-1} + (tI+Q)^{-1}\big)^{-1}\\
   & = 2\left(\dfrac{I}{t} -\dfrac{P}{t(1+t)} + \dfrac{I}{t} -\dfrac{Q}{t(1+t)}\right)^{-1}
   = \left(\dfrac{I}{t} -\dfrac{P+Q}{2t(1+t)}\right)^{-1}\\
   &= 2t(1+t)\left({2(1+t)I -(P+Q)}\right)^{-1}.
\end{align*}
The assumption implies that
$$
2t(1+t)\left({2(1+t)I -(P+Q)}\right)^{-1}\leq \dfrac{2t(t+1)I}{2t+1},
$$
or equivalently,
$$
({2t+1})I \leq {2(1+t)I -(P+Q)}.
$$
This forces
$P+Q \leq I$, or $PQ=0$. If one of $P,Q$ is zero, then we obviously have $PQ=0$. Therefore, $\E_X(s)\E_Y(s)=0$ holds for all $s>0$ and hence we obtain $XY=0$ by Lemma \ref{lem:AW-ortho}.
\end{proof}

The following is  the  $AW^*$-algebra part in Theorem \ref{T:3all} for preservers of the norm of the harmonic mean.

\begin{theorem}\label{thm:gen-hm}
Let $\mathcal{M}, \mathcal{N}$ be $AW^*$-algebras.
Assume that $\phi: \mathcal{M}_{+}^{-1} \to \mathcal{N}_{+}^{-1}$ is a surjective map. Then $\phi$ satisfies
\[
\|\phi(A)!\phi(B)\| = \|A!B\|, \quad A,B\in \mathcal{M}_{+}^{-1},
\]
if and only if it extends to a Jordan $*$-isomorphism from $\mathcal{M}$ onto $\mathcal{N}$.
\end{theorem}
\begin{proof}
It suffices to check the necessity part.
Clearly, $\phi$ preserves the norm.
By Lemma \ref{lem:HM-order}, we see that $\phi$ is an order isomorphism.
By Lemma \ref{lem:App2Ap}, with any fixed $t>0$, the map
$\psi:\mathcal{M}_{+} \to \mathcal{N}_{+}$ defined by
 $\psi(X) = \phi(X+t I) -t I$, $X\in \mathcal{M}_{+}$ is an order isomorphism  preserving the norm.
It follows from Lemma \ref{lem:hm-ortho} that
both $\psi$ and $\psi^{-1}$  preserve the
orthogonality.
Consequently,
Lemma \ref{lem:key} provides a Jordan $*$-isomorphism $J: \mathcal{M}\to \mathcal{N}$ such that
$$
\phi(X+t I) = \psi(X) + tI = J(X) + tI, \quad X\in \mathcal{M}_+.
$$
Observe that we can choose an arbitrary $t>0$ to define $\psi$ and the associated Jordan $*$-isomorphism must be the same.
Indeed, if $t'>0$ is such that $t'> t$ and $t'$ is associated with the Jordan $*$-isomorphism $J'$, then
\begin{eqnarray*}
J'(X)  &=& \phi(X + t'I)- t'I = \phi((X+ (t'-t)I)+tI) - t'I \\
&=& J(X+(t'-t)I) + tI - t'I  = J(X),\quad X\in \mathcal{M}_+.
\end{eqnarray*}
For any $A\in \mathcal{M}_+^{-1}$ we can find $\epsilon >0$ such that $X=A-\epsilon I\geq0$.  It follows
$$
\phi(A) = \phi(X+\epsilon I) = J(X) + \epsilon I = J(A-\epsilon I) + \epsilon I = J(A).
$$
This completes the proof.
\end{proof}

We leave this as an exciting open problem to see if the theorem above remains valid also for general unital $C^*$-algebras.

\section{Nonsurjective maps between positive cones of commutative
$C^*$-algebras which preserve the norm of means of any finitely many elements}\label{S:3}

In this section we consider our preservers in the setting of commutative (but not necessarily unital) $C^*$-algebras;
more concretely, we consider those maps
between positive cones of algebras of continuous functions.  We point out that bijective maps which
preserve the norm of a general homogeneous mean of pairs of functions were characterized in Theorem 3 in \cite{MS15}. In
this section, we are interested in the nonsurjective (meaning that not necessarily surjective) case.

We begin with some necessary notation.
For a locally compact Hausdorff space $\mathrm{X}$, we denote its one-point compactification by $\mathrm{X}\cup \{\infty\}$. Let $C_0(\mathrm{X})$ be the commutative $C^*$-algebra of all complex-valued continuous functions on $\mathrm{X}$ vanishing at infinity equipped with the sup-norm
$\|f\|=\sup\{|f(x)|:x\in \mathrm{X}\}$.  As usual, we also write $C(\mathrm{X})$ for $C_0(\mathrm{X})$ when $\mathrm{X}$ is compact.
Denote by
$$
C_{0}(\mathrm{X})_+ = \{f\in C_0(\mathrm{X}): f(x)\geq 0, x\in \mathrm{X}\}
$$
the positive (semidefinite) cone of $C_0(\mathrm{X})$.

Urysohn's Lemma (see, e.g., \S2.12 in \cite{Rud}) tells us that for any nonempty
sets $K\subset V\subset \mathrm{X}$, with $K$ being compact and $V$ being open, there is a continuous function $f$ on $\mathrm{X}$ with values in $[0,1]$,
which is constant $1$ on $K$ and has compact support contained in $V$.
It follows,  when $\mathrm{X}$ is $\sigma$-compact, the set of strictly positive functions,
$$
C_0(\mathrm{X})_{++} = \{f\in C_0(\mathrm{X}): f(x)> 0, x\in \mathrm{X}\},
$$
is nonempty and  indeed dense in $C_0(\mathrm{X})_+$.
Observe that in the unital case, i.e., when $\mathrm{X}$ is compact,
$C(\mathrm{X})_{++}$ coincides with the set $C(\mathrm{X})_+^{-1}$ of all invertible elements of $C(\mathrm{X})_+$.
Furthermore, define
$
C_0(\mathrm{X})_{+,1}=\{g\in C_{0}(\mathrm{X})_+:\|g\|=1\}
$
and
$
C_0(\mathrm{X})_{++,1}=\{g\in C_{0}(\mathrm{X})_{++}:\|g\|=1\}.
$

In this section,  $F(\mathrm{X})$ denotes a subset of $C_0(\mathrm{X})_{+}$ which is assumed to contain sufficiently many functions to peak on  compact
$G_\delta$ subsets of $\mathrm{X}$. This means that for any nonempty compact  subset $K$ of $\mathrm{X}$, which is
also a $G_\delta$ set, i.e., a countable intersection
of open sets,  we have a function $f\in F(\mathrm{X})$ such that its \textit{peak set}, i.e.,
$
\operatorname{pk}(f)=\{x\in \mathrm{X} : f(x)=\|f\|\},
$
coincides with $K$.
By Urysohn's Lemma, examples of
such families $F(\mathrm{X})$ include $C_0(\mathrm{X})_+$, $C_0(\mathrm{X})_{+,1}$,
and in the case where $\mathrm{X}$ is $\sigma$-compact, $C_0(\mathrm{X})_{++}$ and $C_0(\mathrm{X})_{++,1}$, too.

We call a nonempty subset $A$ of $\mathrm{X}$ a \textit{level set}
if there is a function $f\in C_0(\mathrm{X})_+$ and a scalar $\lambda>0$ such that $A=\{x\in \mathrm{X}: f(x)=\lambda\}$.
Again, by Urysohn's Lemma, it is not difficult to see that level subsets of a locally compact Hausdorff space are exactly
those being nonempty, compact and $G_\delta$.

The following lemma collects those properties of $F(\mathrm{X})$ which we will use in the sequel.

\begin{lemma}\label{lem:F(X)}
Let $\mathrm{X}$ be a locally compact space and let $F(\mathrm{X})\subset C_{0}(\mathrm{X})_+$ contain sufficiently many functions to peak on compact $G_\delta$ subsets
of $\mathrm{X}$.
\begin{enumerate}[(a)]
  \item Let $x_1\in \mathrm{X}$ and $F$ be a closed set in $\mathrm{X}\cup\{\infty\}$ with $x_1\notin F$.
   Then there is $g\in F(\mathrm{X})$ such that
   $$
   x_1\in \operatorname{pk}(g)\quad\text{and}\quad  \operatorname{pk}(g)\cap F =\emptyset.
   $$
  \item Let $f\in C_0(\mathrm{X})_+$ and $x_1\in \mathrm{X}$.  Then there is $g\in F(\mathrm{X})$ such that
  $$
  x_1\in \operatorname{pk}(g) \subset \{x\in \mathrm{X}: f(x)=f(x_1)\}.
  $$

 \item For a net $\{x_\mu\}$ in $\mathrm{X}$ and a point $x_0\in \mathrm{X}\cup\{\infty\}$ we have $x_\mu \to x_0$    if and only if $g(x_\mu) \to g(x_0)$ for all $g\in F(\mathrm{X})$.
\end{enumerate}
\end{lemma}
\begin{proof}
  The assertion (a) is  trivial   since Urysohn's Lemma provides functions in $C_0(\mathrm{X})$ peaking at $x_1$ and vanishing on $F$.

  For (b),
   it is trivial if $f(x_1)>0$ since the level set $\{x\in \mathrm{X}: f(x)=f(x_1)\}$ is compact and $G_\delta$.
  Suppose $f(x_1)=0$.
  Using Urysohn's Lemma we can have     $h\in C_0(\mathrm{X})_+$ with values in $[0,1]$ such that $h(x_1)=1$.
  Let $k(x) = h(x)(1-\min\{1, f(x)\})$
   for all $x\in \mathrm{X}$.  Then $k\in C_0(\mathrm{X})$ and $x_1\in \operatorname{pk}(k)\subset f^{-1}(0)$.  Choose $g\in F(\mathrm{X})$ with
   $\operatorname{pk}(g)=\operatorname{pk}(k)$, and we are done.

  For (c), one direction is plain.  Suppose now that $g(x_\mu)\to g(x_0)$ for every $g\in F(\mathrm{X})$.
We need to show that $x_\mu \to x_0$.
Suppose first that $x_0\in \mathrm{X}$.  Assume  on the contrary that there is an open set $V$ in $\mathrm{X}$ such that
$x_0\in V$ and there is a subnet $\{z_\alpha\}$ of $\{x_\mu\}$ whose elements
are in the closed set
$\mathrm{X}\setminus V$.  We can choose a nonzero $h\in F(\mathrm{X})$ such that $x_0 \in\operatorname{pk}(h)\subset V$ by part (a).
By compactness, $\{z_\alpha\}$ has a subnet $\{w_\beta\}$ which is convergent in $\mathrm{X}\cup\{ \infty\}$. If $w_\beta \to \infty$, then we
have $h(w_\beta)\to 0$ which contradicts $h(x_\mu)\to h(x_0)=\|h\|\neq 0$. If $w_\beta \to w\in \mathrm{X}$, then we have, on one hand,
that $w\in \mathrm{X}\setminus V$ and, on the other hand, that $h(w)=h(x_0)=\|h\|$ yielding $w\in \operatorname{pk}(h)\subset V$,
a contradiction again. Therefore, we have $x_\mu \to x_0$.

Finally, if $x_0=\infty$ then $g(x_\mu)\to 0$ for all $g\in F(\mathrm{X})$.  If  $\{x_\mu\}$ does not converge to $\infty$ then
there is a subnet $\{u_\alpha\}$ converging to $u_0$ in $\mathrm{X}$.  It turns out that $g(u_0)=0$ for all $g\in F(\mathrm{X})$.  This conflicts with
part (a), and thus completes the proof.
\end{proof}

Let now $\mathrm{X},\mathrm{Y}$ be locally compact Hausdorff spaces.  Let
$F(\mathrm{X})\subset C_{0}(\mathrm{X})_+$  contain sufficiently many functions to peak on  compact $G_\delta$ subsets of $\mathrm{X}$. In what
follows we will show that any finitely norm-additive map $T: F(\mathrm{X})\rightarrow C_{0}(\mathrm{Y})_+$ can be written as a generalized composition
operator. Equivalently, if $T$ preserves the norm of the arithmetic mean of any collection of finitely many elements, then $T$ is necessarily
a generalized composition operator.
In fact, along the same line of reasoning, we also obtain
similar results for nonsurjective maps
which preserve the norm of the geometric or the harmonic mean of any collection of finitely many elements.
Actually, we even have a result concerning general power means, namely, Theorem \ref{thm:p-means} stated in the Introduction.

We consider the cases of the three different means separately.
Before going into the details of the proof, we introduce some further notation.
Let  $T: F(\mathrm{X})\rightarrow C_{0}(\mathrm{Y})_+$ be a nonsurjective map preserving the norm of the arithmetic,
the geometric or the harmonic mean of any collection of finitely many elements.
In each case, $T$ preserves the norm, namely, $\|Tf\| = \|f\|$ for any $f\in F(\mathrm{X})$.

For any $x\in \mathrm{X}$ and $f\in F(\mathrm{X})$, beside the peak set
\begin{align}
&\operatorname{pk}(f)=\{z\in \mathrm{X}: f(z)=\|f\|\},\nonumber
\intertext{we define the following sets}
&\operatorname{PKat}(x)=\{g\in F(\mathrm{X}): g(x)=\|g\|\},\nonumber\\
&\PSupp_T(x)=\bigcap_{h\in \operatorname{PKat}(x)}\operatorname{pk}(T h) =\left\{y\in \mathrm{Y}: T(\operatorname{PKat}(x))\subset \operatorname{PKat}(y)\right\},\label{E:Mol}\\
&\Supp_T(x, f)=\PSupp_T(x)\cap\{y\in \mathrm{Y}: (Tf)(y)=f(x)\},\nonumber\\
&\Supp_T(x)=\bigcap_{f\in F(\mathrm{X})}\Supp_T(x, f).
\nonumber
\end{align}
Since $T$ preserves the norm, it is apparent that
\begin{align}\label{eq:yinPsupp}
y\in \PSupp_T(x)\quad \text{if and only if} \quad
(Tf)(y) = f(x)\ \text{whenever}\ f\in F(\mathrm{X})\ \text{peaks at}\ x.
\end{align}
It also holds
\begin{gather}\label{eq:yinsupp}
y\in \Supp_T(x)\quad\text{if and only if}\quad (Tf)(y) = f(x)\ \text{whenever}\ f\in F(\mathrm{X}).
\end{gather}

While $\Supp_T(x)\subseteq \PSupp_T(x)$ for any  $x\in \mathrm{X}$, the following example tells us
that the inclusion can be strict.

\begin{example}\label{eg:Supp-not-PSupp}
  Consider the compact discrete spaces $\mathrm{X}=\{1,2\}$ and $\mathrm{Y}=\{1,2,3\}$.
  Let $F(\mathrm{X})=\{f_1,f_2,f_3\}\subset C(\mathrm{X})_+^{-1}$, in which
  $$
  f_1(1)=f_2(2)=f_3(1)=f_3(2)=1 \quad\text{and}\quad f_1(2)=f_2(1)=1/2.
  $$
  Define $T: F(\mathrm{X})\to C(\mathrm{Y})_+^{-1}$ by
  $$
  (Tf)(1)=f(1),\quad (Tf)(2)=f(2)\quad\text{and}\quad (Tf)(3) = \frac{4f(1)-1}{3}, \quad f\in F(\mathrm{X}).
  $$
  It is easy to check that $T$ preserves the norm of the arithmetic, geometric and harmonic means of any
   finite family of functions from $F(\mathrm{X})$.
  Moreover,
  $$
  \PSupp_T(1) = \{1,3\} \quad\text{and}\quad \PSupp_T(2) = \{2\},
  $$
  while
  $$
  \Supp_T(1) = \{1\} \quad\text{and}\quad \Supp_T(2) = \{2\}.
  $$
  \end{example}

\subsection{Preservers of the norm of the arithmetic mean of any finite collection}

In this subsection, assume that $\mathrm{X},\mathrm{Y}$ are locally compact Hausdorff spaces, and $F(\mathrm{X})\subset C_{0}(\mathrm{X})_+$ is a subset which contains sufficiently many
functions to peak on  compact $G_\delta$ subsets of $\mathrm{X}$.
We consider a (not necessarily additive, positive homogeneous, or surjective)
map $T: F(\mathrm{X})\rightarrow C_{0}(\mathrm{Y})_+$ which preserves the norm of the arithmetic mean of any finite collection of elements of $F(\mathrm{X})$. We note
that $T$ necessarily preserves the norm.

The proof of the fact that $T$ is a generalized composition operator as stated in Theorem \ref{thm:p-means} will be carried out through several observations. We begin with the
following lemma.

\begin{lemma}\label{le302}
For any $x\in \mathrm{X}$, the set
$
\PSupp_T(x)
$
is compact and nonempty.
\end{lemma}
\begin{proof}
Choose $h_1,h_2,\ldots,h_n\in \operatorname{PKat}(x)$. Observe that
\begin{align}\label{ineq307}
\|T h_1+T h_2+\cdots+T h_n\|&=\|h_1+h_2+\cdots+h_n\|= h_1(x)+h_2(x)+\cdots + h_n(x) \notag\\
&=\|h_1\| + \|h_2\| +\cdots +\|h_n\|.
\end{align}
It follows from (\ref{ineq307}) and  the fact  $T$ preserves the norm that there exists a point $y_1\in \mathrm{Y}$ with
$$
(T h_1)(y_1)+(T h_2)(y_1)+\cdots+(T h)_n(y_1) = \|Th_1\| + \|Th_2\| + \cdots +\|Th_n\|.
$$
This forces that
$(T h_i)(y_1)=\|Th_i\|$ for all $1\leq i\leq n$. Consequently,
\begin{align*}
\bigcap_{i=1}^{n}\operatorname{pk}(T h_i)\neq\emptyset.
\end{align*}
Hence, the family $\{\operatorname{pk}(T h)\}_{0\neq h\in \operatorname{PKat}(x)}$ of compact subsets of $\mathrm{Y}$ has the finite intersection property. It follows that the compact set
\begin{align*}
\PSupp_T(x)
=\bigcap_{h\in \operatorname{PKat}(x)}\operatorname{pk}(T h)\neq\emptyset.
\end{align*}
\end{proof}

\begin{lemma}\label{le304}
If $x_1,x_2\in \mathrm{X}$, $x_1\neq x_2$, then we have $\PSupp_T(x_1)\cap \PSupp_T(x_2)=\emptyset$.
\end{lemma}

\begin{proof}
It follows from Lemma \ref{lem:F(X)} that there exists a function $g_1\in F(\mathrm{X})$ such that $g_1(x_1)=\|g_1\|$ and $g_1(x_2)<\|g_1\|$, and another one
$g_2\in F(\mathrm{X})$ such that $g_2(x_2)=\|g_2\|$ and $g_2(x)<\|g_2\|$ for all $x\in \operatorname{pk}(g_1)$. Clearly, $g_1\in \operatorname{PKat}(x_1)$ and $g_2\in \operatorname{PKat}(x_2)$.

Assume that $y\in \PSupp_T(x_1)\cap \PSupp_T(x_2)$. Then $Tg_1, Tg_2\in \operatorname{PKat}(y)$, and hence
\begin{align*}
\|Tg_1+Tg_2\|\geq(Tg_1+Tg_2)(y)=\|Tg_1\| + \|Tg_2\| = \|g_1\| + \|g_2\|.
\end{align*}
On the other hand, $\|Tg_1+Tg_2\|=\|g_1+g_2\|$. Therefore, we would have $\|g_1+g_2\|=\|g_1\| + \|g_2\|$ which forces that $g_1,g_2$ peak at some common point, a contradiction.
\end{proof}

\begin{lemma}\label{le301} For any $g_0, f\in C_{0}(\mathrm{X})_+$, we have
\begin{align}\label{ineq302}
\lim_{n\rightarrow\infty} \left(\|n g_0+f\|-n\|g_0\|\right)=\max \left\{f(x):x\in \operatorname{pk}(g_0)\right\}.
\end{align}
\end{lemma}

\begin{proof}
Let
$$
C=\max \left\{f(x):x\in \operatorname{pk}(g_0)\right\}.
$$
Since $\mathrm{X}$ is locally compact, for any $n\in \mathbb{N}$, there exists $x_n\in \mathrm{X}$ such that
\begin{align}
\label{ineq304}ng_0(x_n)+f(x_n)=\|ng_0+f\|.
\end{align}
From \eqref{ineq304} we infer
\begin{align*}
\lim_{n\rightarrow\infty}g_0(x_n)&=\lim_{n\rightarrow\infty}\frac{\|ng_0+f\|-f(x_n)}{n}
=\lim_{n\rightarrow\infty}\left\|g_0+\frac{f}{n}\right\|-\lim_{n\rightarrow\infty}\frac{f(x_n)}{n}=\|g_0\|.
\end{align*}
This implies that every cluster point $z$ of the sequence $\{x_n\}$ (in the compact space $\mathrm{X}\cup\{\infty\}$) belongs to $\operatorname{pk}(g_0)$.
We further claim that
$f(z)= C$.
Indeed, from (\ref{ineq304})  we deduce that
$$
f(x_n) = \|ng_0 + f\| - ng_0(x_n) \geq n(g_0(x)-g_0(x_n)) + f(x) \geq f(x), \quad x\in \operatorname{pk}(g_0).
$$
Hence, $f(x_n)\geq C$ for all $n\in \mathbb{N}$, and thus $f(z)= C$ holds for every cluster point $z$ of the sequence $\{x_n\}$.
This forces that the bounded scalar sequence  $\{f(x_n)\}$ converges to $C$.
Consequently,
\begin{align*}
f(x_n)&\geq ng_0(x_n)+f(x_n)-n\|g_0\|=\|ng_0+f\|-n\|g_0\|\nonumber\\
&\geq ng_0(z)+f(z)-n\|g_0\|=f(z)=C.
\end{align*}
Letting $n\rightarrow\infty$, we get (\ref{ineq302}).
\end{proof}

\begin{lemma}\label{le306}
Let $x_0\in \mathrm{X}$ and let $g_1,g_2,\ldots,g_m\in \operatorname{PKat}(x_0)$. Then for any $f\in C_{0}(\mathrm{X})_+$ we have
\begin{align}\label{ineq401}
\lim_{n\rightarrow\infty}\left(\left\|n\sum_{j=1}^m g_j+f\right\|-n\sum_{j=1}^m \left\|g_j\right\|\right)=\max\left\{f(x): x\in \bigcap_{j=1}^m \operatorname{pk}(g_j)\right\}.
\end{align}
\end{lemma}

\begin{proof}
Observe that
$$
\left\|\sum_{j=1}^m g_j\right\| = \sum_{j=1}^m \left\|g_j\right\|
\quad\text{and}\quad
\operatorname{pk}\left(\sum_{j=1}^m g_j\right)=\bigcap_{j=1}^m \operatorname{pk}(g_j).
$$
On the other hand, it follows from Lemma \ref{le301} that
\begin{align}\label{ineq402}
\lim_{n\rightarrow\infty}\left(\left\|n\left(\sum_{j=1}^m g_j\right)+f\right\|-n\left\|\sum_{j=1}^m g_j\right\|\right)
=\max\left\{f(z): z\in \operatorname{pk}\left(\sum_{j=1}^m g_j\right)\right\}.
\end{align}
Then (\ref{ineq401}) follows from (\ref{ineq402}).
\end{proof}

Together with \eqref{eq:yinPsupp}, Lemma \ref{le303} tells us
the role of the set $\PSupp_T(x)$, but this is the best one can say as demonstrated in Example \ref{eg:Supp-not-PSupp}.

\begin{lemma}\label{le303}
Let $x_0\in \mathrm{X}$. Then we have
$$
(Tf)(y)\leq f(x_0), \quad  f\in F(\mathrm{X}), y\in \PSupp_T(x_0).
$$
\end{lemma}

\begin{proof}
Fix any $f\in F(\mathrm{X})$.
By Lemma \ref{lem:F(X)}, there is a  function
$g_0\in F(\mathrm{X})$ such that
$$
x_0\in \operatorname{pk}(g_0)\subset \{x\in \mathrm{X}: f(x)=f(x_0)\}.
$$
Lemma \ref{le301} implies that
\begin{align*}
f(x_0)&=\max\{f(x):x\in \operatorname{pk}(g_0)\}\nonumber\\
&=\lim_{n\rightarrow\infty}(\|ng_0+f\|-n\|g_0\|)\nonumber\\
&=\lim_{n\rightarrow\infty}(\|n(Tg_0)+Tf\|-n\|Tg_0\|)\nonumber\\
&=\max\{(Tf)(y):y\in \operatorname{pk}(Tg_0)\}.
\end{align*}
Since $g_0\in \operatorname{PKat}(x_0)$, we have $\PSupp_T(x_0)\subset \operatorname{pk}(Tg_0)$ and
the assertion follows.
\end{proof}

\begin{lemma}\label{le305}
The set $\Supp_T(x_0)$ is compact and nonempty for any $x_0\in \mathrm{X}$.
\end{lemma}

\begin{proof}
Fix  any $f_1,f_2,\ldots,f_l\in F(\mathrm{X})$.  By Lemma \ref{lem:F(X)}, there is a  function $g_0\in F(\mathrm{X})$ such that
\[
x_0\in \operatorname{pk}(g_0)\subset \left\{x\in \mathrm{X}: \sum_{i=1}^lf_i(x)=\sum_{i=1}^l f_i(x_0)\right\}.\]
Hence, for any $g_1,g_2,\ldots,g_m\in \operatorname{PKat}(x_0)$, it follows from Lemma \ref{le306} that
\begin{align*}
\sum_{i=1}^lf_i(x_0)
&=\max\left\{\sum_{i=1}^lf_i(x):x\in \bigcap_{j=0}^m \operatorname{pk}(g_j)\right\}\nonumber\\
&=\lim_{n\rightarrow\infty}\left(\left\|n\sum_{j=0}^m g_j+\sum_{i=1}^lf_i\right\|-n\sum_{j=0}^m \|g_j\|\right)\nonumber\\
&=\lim_{n\rightarrow\infty}\left(\left\|n\sum_{j=0}^m Tg_j+\sum_{i=1}^lTf_i\right\|-n\sum_{j=0}^m \|Tg_j\|\right)\nonumber\\
&=\max\left\{\left(\sum_{i=1}^lTf_i\right)(y):y\in \bigcap_{j=0}^m \operatorname{pk}(Tg_j)\right\}.
\end{align*}
As a result, there is a point $y_0\in \bigcap_{j=0}^m \operatorname{pk}(Tg_j)$ such that
\begin{align*}
\sum_{i=1}^l(Tf_i)(y_0)=\sum_{i=1}^lf_i(x_0).
\end{align*}
Consequently,
\begin{align*}
\left(\bigcap_{j=0}^{m}\operatorname{pk}(T g_j)\right)\bigcap\left\{y\in \mathrm{Y}:\sum_{i=1}^l(Tf_i)(y)=\sum_{i=1}^lf_i(x_0)\right\}\neq\emptyset.
\end{align*}
It follows that the family
$$\left\{\operatorname{pk}(T g)\bigcap\left\{y\in \mathrm{Y}:\sum_{i=1}^l(Tf_i)(y)=\sum_{i=1}^lf_i(x_0)\right\}\right\}_{0\neq g\in \operatorname{PKat}(x_0)}$$
of compact subsets of $\mathrm{Y}$ has the finite intersection property.
Therefore,
\begin{align*}
\bigcap_{g\in \operatorname{PKat}(x_0)}\left(\operatorname{pk}(T g)\bigcap\left\{y\in \mathrm{Y}:\sum_{i=1}^l(Tf_i)(y)=\sum_{i=1}^lf_i(x_0)\right\}\right)\neq\emptyset.
\end{align*}
In other words,
\begin{align*}
\PSupp_T(x_0)\bigcap\left\{y\in \mathrm{Y}:\sum_{i=1}^l(Tf_i)(y)=\sum_{i=1}^lf_i(x_0)\right\}\neq\emptyset.
\end{align*}
This means that there exists a point $y_0\in \PSupp_T(x_0)$ such that $\sum_{i=1}^l(Tf_i)(y_0)=\sum_{i=1}^lf_i(x_0)$.
Moreover, Lemma \ref{le303} implies that $(Tf_i)(y_0)\leq f_i(x_0)$, $i=1,...,l$ and hence we deduce
$$
(Tf_i)(y_0)= f_i(x_0), \quad i=1,2,\ldots,l.
$$

Recall that for any $f\in F(\mathrm{X})$, we have defined
$$
\Supp_T(x_0, f)=\PSupp_T(x_0)\cap\{y\in \mathrm{Y}: Tf(y)=f(x_0)\}.
$$
We have just proved that the family $\{\Supp_T(x_0, f)\}_{f\in F(\mathrm{X})}$ of compact subsets of $\mathrm{Y}$ has the finite intersection
property.  It follows  that
\[\Supp_T(x_0)= \bigcap_{f\in F(\mathrm{X})}\Supp_T(x_0, f)\neq\emptyset.\]
\end{proof}

Lemma \ref{le305} shows that $\Supp_T(x)\neq \emptyset$ for all $x\in \mathrm{X}$.
In view of \eqref{eq:yinsupp},
we can now prove the following statement which is the part of Theorem \ref{thm:p-means} relating the arithmetic mean.

\begin{theorem}\label{th301}
Let $\mathrm{X},\mathrm{Y}$ be locally compact Hausdorff spaces.
Let $F(\mathrm{X})$ be a subset of $C_0(\mathrm{X})_{+}$ containing sufficiently many functions to peak on  compact $G_\delta$ subsets of $\mathrm{X}$.
Let  the map $T: F(\mathrm{X}) \rightarrow C_{0}(\mathrm{Y})_+$ preserve the norm of the arithmetic mean of any finitely many elements of $F(\mathrm{X})$. Then there is
a locally compact subset $\mathrm{Y}_0$ of $\mathrm{Y}$ and a surjective continuous map $\tau: \mathrm{Y}_0\rightarrow \mathrm{X}$ such that
\begin{align}\label{ineq322}
(Tf)(y)= f(\tau (y)), \quad y\in \mathrm{Y}_0, f\in F(\mathrm{X}).
\end{align}
In particular, $T$ is injective. In the case when $\mathrm{X}$ is compact, $\mathrm{Y}_0$ is a closed subset of $\mathrm{Y}$.

If the range $T(F(\mathrm{X}))$ of $T$ also contains sufficiently many functions to peak on  compact $G_\delta$ subsets of $\mathrm{Y}$, then $\mathrm{Y}_0=\mathrm{Y}$ and $\tau$ is a homeomorphism from $\mathrm{Y}$ onto $\mathrm{X}$.
\end{theorem}

\begin{proof}
It follows from  Lemma $\ref{le304}$ that $\Supp_T(x_1)\bigcap \Supp_T(x_2)=\emptyset$ when $x_1,x_2\in \mathrm{X}$ with $x_1\neq x_2$.
Consider the disjoint union
$$
\mathrm{Y}_0=\bigcup_{x\in \mathrm{X}}\Supp_T(x)
$$
of compact sets.
Define $\tau:\mathrm{Y}_0\rightarrow \mathrm{X}$ by
$$
\tau(y)=x\quad\text{if}\quad y\in \Supp_T(x).
$$
Clearly, $\tau$ is a well-defined and surjective map satisfying (\ref{ineq322}).

We show that $\tau$ is continuous.
Pick any net $\{y_\mu\}$ in $\mathrm{Y}_0$ converging to some point $y_0\in \mathrm{Y}_0$. Denote $x_\mu=\tau(y_\mu)$ and $x_0=\tau(y_0)$.
We have $f(x_\mu)=(Tf)(y_\mu)\to (Tf)(y_0)=f(x_0)$ for every $f\in F(\mathrm{X})$.
By Lemma \ref{lem:F(X)}, we have $x_\mu \to x_0$ justifying the continuity of $\tau$.

We next show that the subset $\mathrm{Y}_0$ is locally compact. For each point $y_1$ in $\mathrm{Y}_0$ and any open
neighborhood $U_1$ of $y_1$ in $\mathrm{Y}_0$, it needs to find a compact neighborhood $K_1$
of $y_1$ in $\mathrm{Y}_0$ such that $y_1\in K_1\subset U_1$.
Let $x_1=\tau(y_1)$. Then
$Tf( y_1)=f(x_1)$ holds for all $f\in F(\mathrm{X})$. Fix any nonzero function $f_1$ in $\operatorname{PKat}(x_1)$.
We have that $V_1=\tau^{-1}(\{x\in \mathrm{X}: f_1(x)>\|f_1\|/2\})\cap U_1$ is an open neighbourhood of $y_1$ in $\mathrm{Y}_0$ and it is contained in $U_1$.
Let $V_1=W\cap \mathrm{Y}_0$ for some open
neighborhood $W$ of $y_1$ in $\mathrm{Y}$.
There exists a compact neighborhood $K$ of $y_1$
in $\mathrm{Y}$ such that $y_1\in K\subset W$.
We are going to verify that the set $K_1=K\cap \mathrm{Y}_0$ is a compact neighborhood of $y_1$ in $\mathrm{Y}_0$.
Let $\{y_\mu\}$ be a net in $K_1\subset V_1$. By
passing to a subnet, we can assume that $\{y_\mu\}$ converges to $\widetilde{y}$ in $K$ and we will show that $ \widetilde{y}\in \mathrm{Y}_0$. Let $x_\mu=\tau(y_\mu)$ for any $\mu$. Since $\mathrm{X}\cup\{\infty\}$ is compact, by
passing to a subnet again, we can assume that $\{x_\mu\}$ converges to some $\widetilde{x}$ in $\mathrm{X}$ or $x_\mu\to \infty$. If $x_\mu\to \widetilde{x}$ in $\mathrm{X}$, we have that
\[
(Tf)( \widetilde{y})=\lim Tf(y_\mu) =\lim  f(\tau(y_\mu))=\lim f(x_\mu)=f(\widetilde{x})
\]
holds for all $f$ in $F(\mathrm{X})$. Hence $\widetilde{y}\in \Supp_T(\widetilde{x})$ and thus $\widetilde{y}\in \mathrm{Y}_0$. If $x_\mu\to\infty$, one can derive that
\[
(Tf_1)(\widetilde{y})=\lim Tf_1(y_\mu)=\lim  f_1(\tau(y_\mu))=\lim f_1(x_\mu)=0.
\]
However, the fact that $y_\mu\in V_1$ ensures  $ Tf_1(y_\mu)=f_1(x_\mu)>\|f_1\|/2$ for all $\mu$, and hence $Tf(\widetilde{y})\geq \|f_1\|/2$.
This is a contradiction and we obtain that $\widetilde{y}\in \mathrm{Y}_0$. Therefore, we conclude that $K_1$ is a compact neighborhood of $y_1$ in $\mathrm{Y}_0$ which is contained in $V_1\subset U_1$. This verifies that $\mathrm{Y}_0$ is  locally compact.

We verify the closedness of the
set $\mathrm{Y}_0$ when $\mathrm{X}$ is compact.  Assume that $y'$ is the limit of a net $\{y_\lambda\}$ in $\mathrm{Y}_0$.  Since $\mathrm{X}$ is compact, the net $\{\tau(y_\lambda)\}$
has a cluster point $x'$ in $\mathrm{X}$.  By the continuity of both $Tf$ and $f$, it follows from \eqref{ineq322} that $(Tf)(y') = f(x')$. It plainly follows that $y'\in \Supp_T(x')$ and hence we have $y'\in \mathrm{Y}_0$.

Finally, denoting the range $T(F(\mathrm{X}))$ of $T$ by $F(\mathrm{Y})$, we assume that it contains sufficiently many functions to peak on  compact $G_\delta$ subsets of $\mathrm{Y}$.
We apply the arguments above to the inverse map $T^{-1}: F(\mathrm{Y})\to F(\mathrm{X})$.  Then there is a locally compact subset $\mathrm{X}_0$ of $\mathrm{X}$ and a continuous surjective map
$\sigma: \mathrm{X}_0 \to \mathrm{Y}$ such that
$$
(T^{-1}g)(x) = g(\sigma(x)),\quad g\in F(\mathrm{Y}), x\in \mathrm{X}_0.
$$
Write $f=T^{-1}g\in F(\mathrm{X})$ and rewrite the above formula as
$$
(Tf)(\sigma(x)) = f(x),\quad f\in F(\mathrm{X}), \sigma(x)\in \mathrm{Y}.
$$
This shows that $\sigma(x)\in \Supp_T(x)$ for every $x\in \mathrm{X}_0$ which, by the
surjectivity of $\sigma$, implies that $\mathrm{Y}_0=\mathrm{Y}$. For any $x\in \mathrm{X}_0$ we also obtain that $f(\tau(\sigma(x))=f(x)$ holds for all $f\in F(\mathrm{X})$.
It follows from  Lemma \ref{lem:F(X)} that $F(\mathrm{X})$ separates the points of $\mathrm{X}$.
Hence we deduce $\tau(\sigma(x))=x$ for all $x \in \mathrm{X}_0$.
Similarly, $\mathrm{X}_0=\mathrm{X}$ and $\sigma(\tau(y))=y$  for all $y\in \mathrm{Y}_0$.
We infer that $\tau$ is a homeomorphism from $\mathrm{Y}$ onto $\mathrm{X}$, as asserted.
\end{proof}

By the previous theorem, if we assume that the locally compact Hausdorff spaces $\mathrm{X},\mathrm{Y}$ are $\sigma$-compact,
then any map $T: C_{0}(\mathrm{X})_{++}\rightarrow C_{0}(\mathrm{Y})_{++}$ which is finitely norm-additive or, equivalently, preserves the norm of the arithmetic mean of any finite collection of elements in $C_{0}(\mathrm{X})_{++}$,
 is a generalized composition operator. However,
 $T$ is not necessarily a composition operator.  Indeed,
  it can really happen that $\mathrm{Y}_0$ is a proper subset of $\mathrm{Y}$. In what follows we present an easy example. In fact, what is more, we can also have that $T$ cannot be extended to a norm-additive map between the underlying full function spaces.

\begin{example}\label{eg:not-full}
Let $\mathrm{X}$ be a single point space.  We identify $C(\mathrm{X})_{++}=\, (0,\infty)$.
For $t\in \, (0,1]$, define $Tt\in C([0,1])_{++}$ by $(Tt)(y)=t(1-y/2)$ for all $y\in [0,1]$, and if $t\in \, (1,\infty)$, define $Tt\in C([0,1])_{++}$ by $(Tt)(y)= t$ for all $y\in [0,1]$.
In this way we have a map $T:\, (0,\infty)\,\longrightarrow C([0,1])_{++}$. Clearly, $T$ is finitely norm-additive. We prove that $T$ cannot be extended to a norm-additive map between the full spaces $\mathbb C$ and $C([0,1])$. Indeed, if it were extendable, then it would be an isometry. But this is not the case since $\|T2-T1\|\geq |(T2-T1)(1)|=3/2>1=\|2-1\|$.

Observe that the generalized composition operator form of $T$ given in Theorem \ref{th301} is the following
$$
(Tt)(0)= t, \quad  t\in  C(\mathrm{X})_{++},
$$
where the `good part' $\mathrm{Y}_0$ of $\mathrm{Y}$ is the singleton $\{0\}$, and the continuous surjective map $\tau: \mathrm{Y}_0\to \mathrm{X}$ sends
the unique point $0$ in $\mathrm{Y}_0$ to the single element of $\mathrm{X}$.
\end{example}

In the following example, we see  a finitely norm-additive transformation
$T: C_0(\mathrm{X})_+\to C_0(\mathrm{Y})_+$, for which its generalized composition operator form is defined on a nonclosed subset $\mathrm{Y}_0$ of $\mathrm{Y}$.

\begin{example}[cf. Jeang and Wong {\cite[Example 5]{JW96}}]\label{eg:AM-Y0notclosed}
Let $\mathrm{X}=[0,\infty)$ and $\mathrm{Y}=[-\infty,\infty]$.
Let $T$ be a linear isometry from $C_{0}(\mathrm{X})$ into $C(\mathrm{Y})$ defined
 by
$$
(Tf)(y)=\left\{
\begin{array}{ll}
f(y), & 0 \leq y < \infty,\\
{e^{y}}f(-y), & -\infty < y \leq 0,\\
0, & y = -\infty \text{\ or\ }\infty.
\end{array}\right.
$$
Clearly, $T$ induces a finitely norm-additive map from $C_0(\mathrm{X})_+$ into $C(\mathrm{Y})_+$.
In the notation of Proposition \ref{th102},
$\mathrm{Y}_{0}=[0,\infty)$ is not closed (and not open) in $\mathrm{Y}$.  In
this case,
$\tau(y)=y$ for all $y$ in $\mathrm{Y}_0$, and
$\mathrm{X}$ and $\mathrm{Y}_0$ are homeomorphic.
Note that the range of $T$ does not peak on the compact $G_\delta$ subset $[-1,1]$ of $\mathrm{Y}$, for instance.
\end{example}

\subsection{Preservers of the norm of the geometric mean}

In this subsection we verify the assertion in Theorem \ref{thm:p-means} concerning the geometric mean.
We  follow essentially the same line
of reasoning as in the previous subsection.

Again, in what follows let $\mathrm{X},\mathrm{Y}$ be locally compact Hausdorff spaces, $F(\mathrm{X})\subset C_{0}(\mathrm{X})_+$ be a subset which contains sufficiently many functions to peak on  compact $G_\delta$ subsets of $\mathrm{X}$.
Assume that $T: F(\mathrm{X})\rightarrow C_{0}(\mathrm{Y})_+$ is a (not necessarily additive, positive homogeneous, or surjective) map preserving the norm of the geometric mean of any finitely many elements, i.e.,
$$
\|(Tf_1 Tf_2 \cdots Tf_n)^{1/n}\|= \|(f_1 f_2 \cdots f_n)^{1/n}\|, \quad f_1, f_2, \ldots, f_n\in F(\mathrm{X}), n\in\mathbb{N},
$$
or, equivalently, $T$ preserves the norm of finite products,
$$
\|Tf_1 Tf_2 \cdots Tf_n\|= \|f_1 f_2 \cdots f_n\|,\quad f_1, f_2, \ldots, f_n\in F(\mathrm{X}), n\in\mathbb{N}.
$$
In particular, notice that $T$ preserves the norm, i.e., $\|Tf\|=\|f\|$ holds for all $f\in F(\mathrm{X})$.
We are going to show that $T$ is a generalized composition operator. As already mentioned, the proof will be carried out in a way very similar to what was done in the previous subsection. We keep the notation given there.

We begin with the following observation.

\begin{lemma}\label{leG02}
For any $x\in \mathrm{X}$,  the set
$
\PSupp_T(x)
$
is compact and nonempty.
\end{lemma}

\begin{proof}
Choose nonzero functions $h_1,h_2,\ldots,h_n\in \operatorname{PKat}(x)$. Observe that
\begin{align}\label{ineqG07}
\|T h_1 T h_2 \cdots T h_n\|&=\|h_1 h_2 \cdots h_n\| = h_1(x)h_2(x)\cdots h_n(x) \notag \\
&= \|h_1\|\,\|h_2\|\, \cdots\, \|h_n\|.
\end{align}
Since $T$ is norm preserving, it follows from (\ref{ineqG07})  that there exists a point $y_1\in \mathrm{Y}$ such that
$$
(T h_1)(y_1) (T h_2)(y_1)\cdots (T h_n)(y_1) = \|Th_1\|\,  \|Th_2\|\,  \cdots\, \|Th_n\|.
$$
This forces that
$(T h_i)(y_1)=\|Th_i\|$ holds for all $1\leq i\leq n$. Consequently,
\begin{align*}
\bigcap_{i=1}^{n}\operatorname{pk}(T h_i)\neq\emptyset.
\end{align*}
Hence, the family $\{\operatorname{pk}(T h)\}_{0\neq h\in \operatorname{PKat}(x)}$ of compact subsets of $\mathrm{Y}$ has the finite intersection property.
It follows that the compact set
\begin{align*}
\PSupp_T(x)=\bigcap_{h\in \operatorname{PKat}(x)}\operatorname{pk}(T h)\neq\emptyset.
\end{align*}
\end{proof}

Arguing as in the proof of  Lemma \ref{le304}, but with products replacing sums, one can easily verify the following.

\begin{lemma}\label{leG04}
If $x_1,x_2\in \mathrm{X}$ with $x_1\neq x_2$, then we have $$\PSupp_T(x_1)\cap \PSupp_T(x_2)=\emptyset.$$
\end{lemma}

We continue with the multiplicative analogue of Lemma \ref{le301}.

\begin{lemma}\label{leG01} For any $g_0, f\in C_{0}(\mathrm{X})_+$ with $g_0\neq 0$, we have
\begin{align}\label{ineqG02}
\lim_{n\rightarrow\infty} \frac{\|g_0^nf\|}{\|g_0\|^n}=\max \left\{f(x):x\in \operatorname{pk}(g_0)\right\}.
\end{align}
\end{lemma}
\begin{proof}
Let
$$
C=\max \left\{f(x):x\in \operatorname{pk}(g_0)\right\}.
$$
For any $n\in \mathbb{N}$, there exists $x_n\in \mathrm{X}$ such that
\begin{align}
\label{ineqG04}
g_0(x_n)^n f(x_n)=\|g_0^n f\|.
\end{align}
Assume first that
$f$ does not vanish on the peak set of $g_0$, i.e.,
$C>0$.
Select $x'\in \mathrm{X}$ such that $g_0(x')=\|g_0\|$ and $f(x') > 0$. Then we have
$$
g_0(x_n)^n f(x_n) = \|g_0^n f\| \geq g_0(x')^n f(x') = \|g_0\|^n f(x')>0, \quad n\in\mathbb{N}.
$$
It follows that
\begin{eqnarray*}
\|g_0\| &\geq& g_0(x_n) = \left(\frac{\|g_0^n f\|)}{f(x_n)}\right)^{1/n} \geq \left(\frac{g_0(x')^n f(x'))}{\|f\|}\right)^{1/n}\\
&=& \|g_0\|\left(\frac{f(x'))}{\|f\|}\right)^{1/n},
  \quad n\in\mathbb{N}.
\end{eqnarray*}
Consequently, we have
\begin{align*}
\lim_{n\rightarrow\infty}g_0(x_n)=\|g_0\|.
\end{align*}
This implies that every cluster point $z$ of the sequence $\{x_n\}$ (in the compact space $\mathrm{X}\cup\{\infty\}$) belongs to $\operatorname{pk}(g_0)$.
We further claim that
$f(z)= C$.
Indeed, from (\ref{ineqG04})  we deduce that
\begin{align}\label{eq:GfnK}
f(x_n) = \frac{\|g_0^n f\|}{g_0(x_n)^n} \geq  \frac{\|g_0^n f\|}{\|g_0\|^n} \geq
\sup_{x\in\operatorname{pk}(g_0)} \frac{\|g_0\|^n f(x)}{\|g_0\|^n}=C.
\end{align}
Hence, $f(x_n)\geq C$ for all $n\in \mathbb{N}$.  Using the fact that $z\in \operatorname{pk}(g_0)$, we infer that $f(z)= C$.
Since this is true for any cluster point $z$ of $\{x_n\}$, one can deduce that the bounded scalar sequence  $\{f(x_n)\}$ converges to $C$.
Letting $n\rightarrow\infty$ in \eqref{eq:GfnK}, we get (\ref{ineqG02}).

Finally, we consider the case $C=0$.  It amounts to say  that $f$ vanishes on $\operatorname{pk}(g_0)$.
Pick any positive $\epsilon$. There is an open set $V$ in $\mathrm{X}$ containing $\operatorname{pk}(g_0)$ such that $f(x)<\epsilon$ holds for all $x\in V$. It is easy to see that the supremum of $g$ on the complement of $V$ is strictly less than $\|g_0\|$. Denote this supremum by $s$. We have $g_0^n(x)f(x)\leq \max \{\epsilon \|g_0\|^n, s^n\|f\|\}$ for every $x\in \mathrm{X}$ and $n\in \mathbb N$. It follows that
\[
\frac{\|g_0^n f\|}{\|g_0\|^n} \leq \max\{\epsilon, (s/\|g_0\|)^n \|f\|\}
\]
and the number on the right hand side equals $\epsilon$ for large enough $n$. This proves that
\[
\lim_{n\to \infty} \frac{\|g_0^nf\|}{\|g_0\|^n} =0,
\]
which completes the proof of the lemma.
\end{proof}

\begin{lemma}\label{leG06}
Let $x_0\in \mathrm{X}$ and let $g_1,g_2,\ldots,g_m\in \operatorname{PKat}(x_0)$ be all nonzero functions. Then for any $f\in C_{0}(\mathrm{X})_+$ we have
\begin{align}\label{ineqG401}
\lim_{n\rightarrow\infty}\frac{\left\|\left(\prod_{j=1}^m g_j\right)^n f\right\|}{\prod_{j=1}^m \left\|g_j\right\|}=\max\left\{f(x): x\in \bigcap_{j=1}^m \operatorname{pk}(g_j)\right\}.
\end{align}
\end{lemma}

\begin{proof}
Since $\prod_{j=1}^m g_j\in \operatorname{PKat}(x_0)$, we have
$$
\left\|\prod_{j=1}^m g_j\right\| = \prod_{j=1}^m \left\|g_j\right\|\quad\text{and}\quad
\operatorname{pk}\left(\prod_{j=1}^m g_j\right)=\bigcap_{j=1}^m \operatorname{pk}(g_j).
$$
On the other hand, it follows from Lemma \ref{leG01} that
\begin{align*}
\lim_{n\rightarrow\infty}\frac{\left\|\left(\prod_{j=1}^m g_j\right)^n f\right\|}{\left\|\prod_{j=1}^m g_j\right\|}
=\max\left\{f(z): z\in \operatorname{pk}\left(\prod_{j=1}^m g_j\right)\right\}.
\end{align*}
Now (\ref{ineqG401}) follows.
\end{proof}

\begin{lemma}\label{leG03}
Let $x_0\in \mathrm{X}$. Then
$$
(Tf)(y)\leq f(x_0), \quad f\in F(\mathrm{X}), y\in \PSupp_T(x_0).
$$
\end{lemma}

\begin{proof}
Fix any $f\in F(\mathrm{X})$.
It follows from Lemma \ref{lem:F(X)} that there is a nonzero function
$g_0\in F(\mathrm{X})$ such that
$$
x_0\in\operatorname{pk}(g_0)\subset \{x\in \mathrm{X}: f(x)=f(x_0)\}.
$$
Lemma \ref{leG01} implies that
\begin{align*}
f(x_0)&=\max\{f(x):x\in \operatorname{pk}(g_0)\}\nonumber\\
&=\lim_{n\rightarrow\infty}\frac{\|g_0^n f\|}{\|g_0\|^n}\nonumber\\
&=\lim_{n\rightarrow\infty}\frac{\|(Tg_0)^n Tf\|}{\|Tg_0\|^n}\nonumber\\
&=\max\{(Tf)(y):y\in \operatorname{pk}(Tg_0)\}.
\end{align*}
Since $g_0\in \operatorname{PKat}(x_0)$, we have $\PSupp_T(x_0)\subset \operatorname{pk}(Tg_0)$.
The assertion follows.
\end{proof}

\begin{lemma}\label{leG05}
The set $\Supp_T(x_0)$ is compact and nonempty for any $x_0\in \mathrm{X}$.
\end{lemma}
\begin{proof}
Let $x_0\in \mathrm{X}$.
Let us
fix  any $f_1,f_2,\ldots,f_l\in F(\mathrm{X})$ with $\prod_{i=1}^l f_i(x_0)\neq 0$.
By Lemma \ref{lem:F(X)}, we can choose a nonzero function $g_0\in F(\mathrm{X})$ such that
\[x_0\in \operatorname{pk}(g_0)\subset \left\{x\in \mathrm{X}: \prod_{i=1}^lf_i(x)=\prod_{i=1}^l f_i(x_0)\right\}.\]
Hence, for any nonzero $g_1,g_2,\ldots,g_m\in \operatorname{PKat}(x_0)$, it follows from Lemma \ref{leG06} that
\begin{align*}
\prod_{i=1}^lf_i(x_0)
&=\max\left\{\prod_{i=1}^lf_i(x):x\in \bigcap_{j=0}^m \operatorname{pk}(g_j)\right\}
\nonumber\\
&=\lim_{n\rightarrow\infty}\frac{\left\|\left(\prod_{j=0}^m g_j\right)^n \prod_{i=1}^lf_i\right\|}
{\left(\prod_{j=0}^m \|g_j\|\right)^n}\nonumber\\
&=\lim_{n\rightarrow\infty}\frac{\left\|\left(\prod_{j=0}^m Tg_j\right)^n \prod_{i=1}^l Tf_i\right\|}
{\left(\prod_{j=0}^m \|Tg_j\|\right)^n}\nonumber\\
&=\max\left\{\left(\prod_{i=1}^lTf_i\right)(y):y\in \bigcap_{j=0}^m \operatorname{pk}(Tg_j)\right\}.
\end{align*}
As a result, there is a $y_0\in \bigcap_{j=0}^m \operatorname{pk}(Tg_j)$ such that
\begin{align*}
\prod_{i=1}^l(Tf_i)(y_0)=\prod_{i=1}^lf_i(x_0).
\end{align*}
Consequently, we have
\begin{align*}
\left(\bigcap_{j=0}^{m}\operatorname{pk}(T g_j)\right)\bigcap\left\{y\in \mathrm{Y}:\prod_{i=1}^l(Tf_i)(y)=\prod_{i=1}^lf_i(x_0)\right\}\neq\emptyset.
\end{align*}
Therefore, the family
$$\left\{\operatorname{pk}(T g)\bigcap\left\{y\in \mathrm{Y}:\prod_{i=1}^l(Tf_i)(y)=\prod_{i=1}^lf_i(x_0)\right\}\right\}_{0\neq g\in \operatorname{PKat}(x_0)}$$
of compact subsets of $\mathrm{Y}$ has the finite intersection property.
It follows that
\begin{align*}
\bigcap_{g\in \operatorname{PKat}(x_0)}\left(\operatorname{pk}(T g)\bigcap\left\{y\in \mathrm{Y}:\prod_{i=1}^l(Tf_i)(y)=\prod_{i=1}^lf_i(x_0)\right\}\right)\neq\emptyset,
\end{align*}
that is,
\begin{align*}
\PSupp_T(x_0)\bigcap\left\{y\in \mathrm{Y}:\prod_{i=1}^l(Tf_i)(y)=\prod_{i=1}^lf_i(x_0)\right\}\neq\emptyset.
\end{align*}
This means that there exists a point $y_0\in \PSupp_T(x_0)$ such that $\prod_{i=1}^l(Tf_i)(y_0)=\prod_{i=1}^lf_i(x_0)$.
Moreover, Lemma \ref{leG03} implies that $(Tf_i)(y_0)\leq f_i(x_0)$, and hence we infer
$$
(Tf_i)(y_0)= f_i(x_0), \quad i=1,2,\ldots,l.
$$
It follows that the family of the compact sets
$$
\Supp_T(x_0, f)=\PSupp_T(x_0)\cap\{y\in \mathrm{Y}: Tf(y)=f(x_0)\}, \; f\in F(\mathrm{X})\,\text{with}\,f(x_0)\neq 0,
$$
has the finite intersection
property.  Therefore,
\[ \bigcap_{f\in F(\mathrm{X}),\, f(x_0)\neq 0}\Supp_T(x_0, f)\neq\emptyset.\]

Finally, we  note that by Lemma \ref{leG03}, $f(x_0)=0$ implies $Tf(y)=0$ for all $y\in \PSupp_T(x_0)$.
Hence,
$$
\Supp_T(x_0)=\bigcap_{f\in F(\mathrm{X})}\Supp_T(x_0, f)=\bigcap_{f\in F(\mathrm{X}),\, f(x_0)\neq 0}\Supp_T(x_0, f)\neq\emptyset.
$$
\end{proof}

We have seen that $\Supp_T(x)\neq \emptyset$ holds for all $x\in \mathrm{X}$.
In view of \eqref{eq:yinsupp},
we can formulate the following statement which is the part of Theorem \ref{thm:p-means} relating the geometric mean.

\begin{theorem}\label{thG01}
Let $\mathrm{X},\mathrm{Y}$ be locally compact Hausdorff spaces and
let $F(\mathrm{X})$ be a subset of $C_0(\mathrm{X})_{+}$ containing sufficiently many functions to peak on  compact $G_\delta$ subsets of $\mathrm{X}$.
Let  $T: F(\mathrm{X}) \rightarrow C_{0}(\mathrm{Y})_+$ be a map which preserves the norm of the geometric mean of any finite collections of elements in $F(\mathrm{X})$. Then there is a locally compact subset $\mathrm{Y}_0$ of $\mathrm{Y}$ and a surjective continuous map $\tau: \mathrm{Y}_0\rightarrow \mathrm{X}$ such that
\begin{align*}
(Tf)(y)= f(\tau (y)), \quad y\in \mathrm{Y}_0, f\in F(\mathrm{X}).
\end{align*}
In particular, $T$ is injective. In the case when $\mathrm{X}$ is compact, $\mathrm{Y}_0$ is a closed subset of $\mathrm{Y}$.

If the range $T(F(\mathrm{X}))$ of $T$ also contains sufficiently many functions to peak on  compact $G_\delta$ subsets of $\mathrm{Y}$, then we have $\mathrm{Y}_0=\mathrm{Y}$ and $\tau$ is a homeomorphism from $\mathrm{Y}$ onto $\mathrm{X}$.
\end{theorem}
\begin{proof}
One can simply follow the argument given in the proof of Theorem \ref{th301}.
\end{proof}

In the following, modifying Example \ref{eg:2notn},
we see that a map preserving the norm of the geometric mean of any pair of strictly positive continuous functions might
not be a generalized composition operator.

\begin{example}\label{eg:2GM-not-CO}
Let $\mathrm{X}$ be a single point space, and identify the cone $C(\mathrm{X})_{++}$ of all strictly
positive functions on $\mathrm{X}$ with the interval $(0,\infty)$.
Let $A_1=\,(1,2]$, $A_2=\, (2,3]$ and $A_3=\, (3,\infty)$.
Define $\psi_i:\, (0,\infty)\,\longrightarrow\, (\!-\infty,\infty)$ by
\begin{align}\psi_i(t) =\left\{\begin{array}{ll}
\dfrac{\log t}{2},&\,\,\,\,  t\in A_i,\\
\log t,&\,\,\,\, t\notin A_i,\end{array} \quad i=1,2,3.
\right. \nonumber\end{align}
We construct a map $S:\, (0,\infty)\,\longrightarrow C([0,1])_{++}$ by
 \begin{align}(St)(y)=\left\{\begin{array}{ll}
\exp[(1-2y)\psi_1(t)+2y\psi_2(t)],&\,\,\,\,  y\in[0,1/2],\\
\exp[(2-2y)\psi_2(t)+(2y-1)\psi_3(t)],&\,\,\,\, y\in \, (1/2,1],\end{array}\right.\nonumber\end{align}
for all $t>1$, and $(St)(y)= t$ for all $0\leq y\leq 1$ and $0<t\leq 1$.
Since $\exp$ is  strictly increasing, $S$ preserves the norm of the geometric mean of any pair
of strictly positive functions.  It is plain that $S$ is not a generalized composition operator.
\end{example}

Also observe that the map in Example \ref{eg:AM-Y0notclosed} preserves the norm of
the geometric mean of any finite collection of elements,
but its generalized composition operator form is defined on a nonclosed subset $\mathrm{Y}_0$ of $\mathrm{Y}$.

\subsection{Preservers of the norm of the harmonic mean}

In this subsection
we verify the assertion in Theorem \ref{thm:p-means} concerning the harmonic mean. Again, the basic steps of the proof follow the pattern given in the previous two subsections.
However, we assume here that $\mathrm{X},\mathrm{Y}$ are both compact Hausdorff spaces and $F(\mathrm{X})$ is a subset of $C(\mathrm{X})_+^{-1}$ containing
sufficiently many functions to peak on  compact $G_\delta$ subsets of $\mathrm{X}$.
Let $T: F(\mathrm{X})\rightarrow C(\mathrm{Y})_+^{-1}$ be a (not necessarily additive, positive homogeneous, or surjective)
map which preserves the norm of the harmonic mean of any finite collection of elements in $F(\mathrm{X})$.  In other words, we assume that
$$
\left\|\left(\dfrac{(Tf_1)^{-1}+ (Tf_2)^{-1} +\cdots +  (Tf_n)^{-1}}{n}\right)^{-1}\right\|
= \left\|\left(\dfrac{f_1^{-1}+ f_2^{-1} +\cdots +  f_n^{-1}}{n}\right)^{-1}\right\|,
$$
or, equivalently,
$$
\|((Tf_1)^{-1}+ (Tf_2)^{-1} +\cdots +  (Tf_n)^{-1})^{-1}\|=
\|(f_1^{-1}+ f_2^{-1} +\cdots +  f_n^{-1})^{-1}\|
$$
holds for all $f_1, f_2, \ldots, f_n\in F(\mathrm{X})$ and $n\in\mathbb{N}$.
Here, $f^{-1}=1/f$ is the reciprocal function, i.e., the multiplicative inverse of the strictly positive function $f$.
Observe again that $T$ necessarily preserves the norm, that is, $\|Tf\|=\|f\|$ holds for all $f\in F(\mathrm{X})$.

\begin{lemma}\label{leH02}
For any $x\in \mathrm{X}$,  the set
$
\PSupp_T(x)
$
is compact and nonempty.
\end{lemma}
\begin{proof}
Choose $h_1,h_2,\ldots,h_n\in \operatorname{PKat}(x)$. Observe that
\begin{align}\label{ineqH07}
&\|((Th_1)^{-1}+ (Th_2)^{-1} +\cdots +  (Th_n)^{-1})^{-1}\| =
\|(h_1^{-1}+ h_2^{-1} +\cdots +  h_n^{-1})^{-1}\| \nonumber\\
=\quad &\frac{1}{\frac{1}{h_1(x)} + \frac{1}{h_2(x)} +\cdots + \frac{1}{h_n(x)}}
= \left(\|(h_1\|^{-1}+ \|h_2\|^{-1} +\cdots +  \|h_n^{-1}\|\right)^{-1}.
\end{align}
It follows from (\ref{ineqH07}) and from the norm preserving property of $T$ that there exists a point $y_1\in \mathrm{Y}$ with
\begin{align*}
&\ \frac{1}{\frac{1}{(Th_1)(y_1)} + \frac{1}{(Th_2)(y_1)} +\cdots + \frac{1}{(Th_n)(y_1)}}\\
=&\ \|((Th_1)^{-1}+ (Th_2)^{-1} +\cdots +  (Th_n)^{-1})^{-1}\| \nonumber\\
=&\  \left(\|(Th_1\|^{-1}+ \|Th_2\|^{-1} +\cdots +  \|Th_n^{-1}\|\right)^{-1}.
\end{align*}
This forces that
$(T h_i)(y_1)=\|Th_i\|$ holds for all $1\leq i\leq n$. Consequently,
\begin{align*}
\bigcap_{i=1}^{n}\operatorname{pk}(T h_i)\neq\emptyset.
\end{align*}
Hence the family $\{\operatorname{pk}(T h)\}_{h\in \operatorname{PKat}(x)}$ of compact
subsets of $\mathrm{Y}$ has the finite intersection property. It follows that the compact set
\begin{align*}
\PSupp_T(x)=\bigcap_{h\in \operatorname{PKat}(x)}\operatorname{pk}(T h)\neq\emptyset.
\end{align*}
\end{proof}

Arguing as in the proof of  Lemma \ref{le304} but with harmonic means replacing arithmetic means, we have the following.

\begin{lemma}\label{leH04}
If $x_1,x_2\in \mathrm{X}$, $x_1\neq x_2$, then $\PSupp_T(x_1)\cap \PSupp_T(x_2)=\emptyset$.
\end{lemma}

The following lemma is an analogue of the former Lemmas \ref{le301} and \ref{leG01}.

\begin{lemma}\label{leH01} For any $g_0, f\in C_{0}(\mathrm{X})_+^{-1}$, we have
\begin{align}\label{ineqH02}
\lim_{n\rightarrow\infty} (\|(ng_0^{-1}+f^{-1})^{-1}\|^{-1}-n\|g_0\|^{-1})^{-1}=\max \left\{f(x):x\in \operatorname{pk}(g_0)\right\}.
\end{align}
\end{lemma}

\begin{proof}
Let
$$
C=\max \left\{f(x):x\in \operatorname{pk}(g_0)\right\}.
$$
Since $\mathrm{X}$ is compact, for any $n\in \mathbb{N}$, there exists $x_n\in \mathrm{X}$ such that
\begin{align}
\label{ineqH04}
\frac{1}{\frac{n}{g_0(x_n)}+\frac{1}{f(x_n)}}=\|(ng_0^{-1}+f^{-1})^{-1}\|.
\end{align}
Observe that
\begin{eqnarray*}
\frac{1}{g_0(x_n)} &=& \frac{\|(ng_0^{-1}+f^{-1})^{-1}\|^{-1} - \frac{1}{f(x_n)}}{n}\\
&=& \left\|\left(g_0^{-1}+\frac{f^{-1}}{n}\right)^{-1}\right\|^{-1} -  \frac{1}{nf(x_n)}, \quad n\in\mathbb{N}.
\end{eqnarray*}
Since $f$ has a positive minimum on $\mathrm{X}$, we have
\begin{align*}
\lim_{n\to \infty} g_0(x_n) = \|g_0\|.
\end{align*}
Consequently, every cluster point $z$ of the sequence $\{x_n\}$ belongs to $\operatorname{pk}(g_0)$.
We further claim that
$f(z)= C$.
Indeed, it follows from (\ref{ineqH04})   that for any $x\in \operatorname{pk}(g_0)$ we have
\begin{align*}
\frac{1}{\frac{n}{g_0(x_n)} + \frac{1}{f(x_n)}}\ =\ \|(ng_0^{-1}+f^{-1})^{-1}\|
\ \geq\ \frac{1}{\frac{n}{g_0(x)} + \frac{1}{f(x)}}
\ =\ \frac{1}{\frac{n}{\|g_0\|} + \frac{1}{f(x)}}.
\end{align*}
Note that $\|g_0\| \geq g_0(x_n)>0$ and $\frac{n}{\|g_0\|} + \frac{1}{f(x)}>0$ holds for all $n\in\mathbb{N}$. Hence,
$$
\frac{1}{f(x)} -  \frac{1}{f(x_n)} \geq n \left(\frac{1}{g_0(x_n)}-\frac{1}{\|g_0\|}\right) \geq 0, \quad n\in \mathbb{N}.
$$
Taking the maximum of $f$ over the set $\operatorname{pk}(g_0)$, we deduce
$$
f(x_n) \geq C, \quad n\in \mathbb{N}.
$$
It follows that
$f(z)\geq C$, and thus $f(z)= C$.
Since this holds for every cluster point $z$ of the sequence $\{x_n\}$, the bounded sequence  $\{f(x_n)\}$ of positive numbers converges to $C$.
Consequently, we infer
\begin{align*}
\frac{1}{f(x_n)}
&=  \frac{1}{\|(ng_0^{-1}+f^{-1})^{-1}\|} - \frac{n}{g_0(x_n)}
\leq \inf_{z\in \operatorname{pk}(g_0)} \left\{\frac{n}{g_0(z)}+\frac{1}{f(z)}\right\} - \frac{n}{g_0(x_n)} \\
&=  \frac{n}{\|g_0\|}+\frac{1}{C} - \frac{n}{g_0(x_n)},
\end{align*}
and thus
$$
0\leq n\left(\frac{1}{g_0(x_n)} - \frac{1}{\|g_0\|}\right) \leq \frac{1}{C} - \frac{1}{f(x_n)}
$$
for  all $n\in\mathbb{N}$.
It follows that
$$
\lim_{n\to\infty}  n\left(\frac{1}{g_0(x_n)} - \frac{1}{\|g_0\|}\right)=0.
$$
Letting $n\rightarrow\infty$ in
\begin{align*}
\frac{1}{f(x_n)}
=  \frac{1}{\|(ng_0^{-1}+f^{-1})^{-1}\|} - \frac{n}{\|g_0\|} + n\left(\frac{1}{\|g_0\|} - \frac{1}{g_0(x_n)}\right),
\end{align*}
we get (\ref{ineqH02}).
\end{proof}

\begin{lemma}\label{leG16}
Let $x_0\in \mathrm{X}$ and let $g_1,g_2,\ldots,g_m\in \operatorname{PKat}(x_0)$.
Then for any $f\in C(\mathrm{X})_+^{-1}$ we have
\begin{align}\label{ineqH401}
&\lim_{n\rightarrow\infty} \left(\left\|\left(n\left(\sum_{j=1}^m g_j^{-1}\right) +f^{-1}\right)^{-1}\right\|^{-1} -
n\left(\sum_{j=1}^m \left\|g_j\right\|^{-1}\right) \right)^{-1}\nonumber\\
&=\max\left\{f(x): x\in \bigcap_{j=1}^m \operatorname{pk}(g_j)\right\}.
\end{align}
\end{lemma}

\begin{proof}
Since $(\sum_{j=1}^m g_j^{-1})^{-1}\in \operatorname{PKat}(x_0)$, we have
$$
\left\|\left(\sum_{j=1}^m g_j^{-1}\right)^{-1}\right\| = \left(\sum_{j=1}^m \left\|g_j\right\|^{-1}\right)^{-1}\quad\text{and}\quad
\operatorname{pk}\left(\left(\sum_{j=1}^m g_j^{-1}\right)^{-1}\right)=\bigcap_{j=1}^m \operatorname{pk}(g_j).
$$
On the other hand, it follows from Lemma \ref{leH01} that
\begin{align*}
&\lim_{n\rightarrow\infty} \left(\left\|\left(n\left(\sum_{j=1}^m g_j^{-1}\right) +f^{-1}\right)^{-1}\right\|^{-1} -
n\left\|\left(\sum_{j=1}^m g_j^{-1}\right)^{-1}\right\|^{-1} \right)^{-1}\\
&=\max\left\{f(x): x\in \operatorname{pk}\left(\left(\sum_{j=1}^m g_j^{-1}\right)^{-1}\right)\right\}.
\end{align*}
After this, (\ref{ineqH401}) follows.
\end{proof}

\begin{lemma}\label{leH03}
Let $x_0\in \mathrm{X}$. Then we have
$$
(Tf)(y)\leq f(x_0), \quad  f\in F(\mathrm{X}), y\in \PSupp_T(x_0).
$$
\end{lemma}

\begin{proof}
Fix any $f\in F(\mathrm{X})$. By Lemma \ref{lem:F(X)}, there is a  function
$g_0\in F(\mathrm{X})$ such that
$$
x_0\in\operatorname{pk}(g_0)\subset\{x\in \mathrm{X}: f(x)=f(x_0)\}.
$$
Lemma \ref{leH01} implies that
\begin{align*}
f(x_0)&=\max\{f(x):x\in \operatorname{pk}(g_0)\}\nonumber\\
&=\lim_{n\rightarrow\infty}(\|(ng_0^{-1}+f^{-1})^{-1}\|^{-1}-n\|g_0\|^{-1})^{-1}\nonumber\\
&=\lim_{n\rightarrow\infty}(\|(n(Tg_0)^{-1}+(Tf)^{-1})^{-1}\|^{-1}-n\|Tg_0\|^{-1})^{-1}\nonumber\\
&=\max\{(Tf)(y):y\in \operatorname{pk}(Tg_0)\}.
\end{align*}
Since $g_0\in \operatorname{PKat}(x_0)$, we have $\PSupp_T(x_0)\subset \operatorname{pk}(Tg_0)$.
The assertion follows.
\end{proof}

\begin{lemma}\label{leH05}
The set $\Supp_T(x)$ is nonempty for any $x\in \mathrm{X}$.
\end{lemma}

\begin{proof}
Let $x_0\in \mathrm{X}$.
Fix functions $f_1,\ldots f_l\in F(\mathrm{X})$.
By Lemma \ref{lem:F(X)}, we can choose a  function $g_0\in F(\mathrm{X})$ such that
\[
x_0\in \operatorname{pk}(g_0)\subset\left\{x\in \mathrm{X}: \sum_{i=1}^l f_i(x)^{-1}=\sum_{i=1}^l f_i(x_0)^{-1}\right\}.\]
Hence, for any $g_1,g_2,\ldots,g_m\in \operatorname{PKat}(x_0)$, it follows from Lemma \ref{leG16} that
\begin{align*}
&\left(\sum_{i=1}^l f_i(x_0)^{-1}\right)^{-1}\\
&=\max\left\{\left(\sum_{i=1}^l f_i(x)^{-1}\right)^{-1} :x\in \bigcap_{j=0}^m \operatorname{pk}(g_j)\right\}\nonumber\\
&=\lim_{n\rightarrow\infty} \left(\left\|\left(n\left(\sum_{j=0}^m g_j^{-1}\right) +\sum_{i=1}^l f_i^{-1}\right)^{-1}\right\|^{-1} -
n\left(\sum_{j=1}^m \left\|g_j\right\|^{-1}\right) \right)^{-1} \nonumber\\
&=\lim_{n\rightarrow\infty}\left(\left\|\left(n\left(\sum_{j=0}^m (Tg_j)^{-1}\right) +\sum_{i=1}^l Tf_i^{-1}\right)^{-1}\right\|^{-1} -
n\left(\sum_{j=1}^m \left\|Tg_j\right\|^{-1}\right) \right)^{-1} \nonumber\\
&=\max\left\{\left(\sum_{i=1}^l (Tf_i)(y)^{-1}\right)^{-1} :y\in \bigcap_{j=0}^m \operatorname{pk}(Tg_j)\right\}.
\end{align*}
It follows that there is a point $y_0\in \bigcap_{j=0}^m \operatorname{pk}(Tg_j)$ such that
\begin{align*}
\sum_{i=1}^l(Tf_i)(y_0)^{-1}=\sum_{i=1}^lf_i(x_0)^{-1}.
\end{align*}
Consequently, we obtain
\begin{align*}
\left(\bigcap_{j=0}^{m}\operatorname{pk}(T g_j)\right)\bigcap\left\{y\in \mathrm{Y}:\sum_{i=1}^l(Tf_i)(y)^{-1}=\sum_{i=1}^lf_i(x_0)^{-1}\right\}\neq\emptyset.
\end{align*}
Therefore, the family
$$\left\{\operatorname{pk}(T g)\bigcap\left\{y\in \mathrm{Y}:\sum_{i=1}^l(Tf_i)(y)^{-1}=\sum_{i=1}^lf_i(x_0)^{-1}\right\}\right\}_{g\in \operatorname{PKat}(x_0)}$$
of compact subsets of $\mathrm{Y}$ has the finite intersection property.
It follows that
\begin{align*}
\bigcap_{g\in \operatorname{PKat}(x_0)}\left(\operatorname{pk}(T g)\bigcap\left\{y\in \mathrm{Y}:\sum_{i=1}^l(Tf_i)(y)^{-1}=\sum_{i=1}^lf_i(x_0)^{-1}\right\}\right)\neq\emptyset,
\end{align*}
that is,
\begin{align*}
\PSupp_T(x_0)\bigcap\left\{y\in \mathrm{Y}:\sum_{i=1}^l(Tf_i)(y)^{-1}=\sum_{i=1}^lf_i(x_0)^{-1}\right\}\neq\emptyset.
\end{align*}
This means that there exists a point $y_0\in \PSupp_T(x_0)$ such that $\sum_{i=1}^l(Tf_i)(y_0)^{-1}=\sum_{i=1}^lf_i(x_0)^{-1}$.
Moreover, Lemma \ref{leH03} implies that $(Tf_i)(y_0)\leq f_i(x_0)$, and hence we obtain
$$
(Tf_i)(y_0)= f_i(x_0), \quad i=1,2,\ldots,l.
$$
In other words, the family  of compact sets
$$
\Supp_T(x_0, f)=\PSupp_T(x_0)\cap\{y\in \mathrm{Y}: Tf(y)=f(x_0)\}, \quad f\in F(\mathrm{X}),
$$
 has the finite intersection
property.  It follows  that
\[\Supp_T(x_0)=\bigcap_{f\in F(\mathrm{X})}\Supp_T(x_0, f)\neq\emptyset.\]
\end{proof}

As we did with Theorems \ref{th301} and \ref{thG01}, we can verify the following result which is the part of Theorem \ref{thm:p-means}
relating the harmonic mean.

\begin{theorem}\label{thH01}
Let $\mathrm{X},\mathrm{Y}$ be compact Hausdorff spaces.
Let $F(\mathrm{X})$ be a subset of $C(\mathrm{X})_{+}^{-1}$ containing sufficiently many functions to peak on  compact $G_\delta$ subsets of $\mathrm{X}$.
Let  $T: F(\mathrm{X}) \rightarrow C_{0}(\mathrm{Y})_{+}^{-1}$ be a map which preserves the norm of the harmonic mean of any finite collection of elements in $F(\mathrm{X})$.
Then there is a closed subset $\mathrm{Y}_0$ of $\mathrm{Y}$ and a surjective continuous map $\tau: \mathrm{Y}_0\rightarrow \mathrm{X}$ such that
\begin{align*}
(Tf)(y)= f(\tau (y)), \quad y\in \mathrm{Y}_0, f\in F(\mathrm{X}).
\end{align*}
In particular, $T$ is injective.

If the range $T(F(\mathrm{X}))$ of $T$ also  contains sufficiently many functions to peak on  compact $G_\delta$ subsets of $\mathrm{Y}$, then $\mathrm{Y}_0=\mathrm{Y}$ and $\tau$ is a homeomorphism from $\mathrm{Y}$ onto $\mathrm{X}$.
\end{theorem}

In the following,
modifying Example \ref{eg:2notn}, we find a map preserving the norm of the harmonic mean of any pair of strictly positive continuous
functions which is not a generalized composition operator.

\begin{example}\label{eg:2HM-not-CO}
Let $\mathrm{X}$ be a single point space and identify the cone $C(\mathrm{X})_{++}$ of all strictly
positive functions on $\mathrm{X}$ with the interval $(0,\infty)$.
Let $A_1=\, (0,1]$, $A_2=\, (1,2]$ and $A_3=\, (2,\infty)$.
Define $\psi_i:\, (0,\infty)\,\longrightarrow\, (0,\infty)$ by
\begin{align}\psi_i(t) =\left\{\begin{array}{ll}
2t,&\,\,\,\,  t\in A_i,\\
t,&\,\,\,\, t\notin A_i,\end{array} \quad i=1,2,3.
\right. \nonumber\end{align}
We construct a map $V:\, (0,\infty)\,\longrightarrow C([0,1])_{++}$ by
 \begin{align}
 (Vt)(y)=\left\{\begin{array}{ll}
(1-2y)\psi_1(t)+2y\psi_2(t),&\quad  y\in[0,1/2],\\
(2-2y)\psi_2(t)+(2y-1)\psi_3(t),&\quad y\in \, (1/2,1],\end{array}\right.\nonumber
\end{align}
for all $t\in \, (0,\infty)$.
Define a new map $R: (0,\infty)\, \longrightarrow C([0,1])_{++}$ by
$$
Rt= \dfrac{1}{{V(1/t)}}, \quad t\in (0,\infty).
$$
Since
$$
\min_{y\in [0,1]} \{V(t_1)(y) + V(t_2)(y)\} = t_1 + t_2,
$$
we see that
$$
\left\|\frac{1}{R(t_1)^{-1} + R(t_2)^{-1}}\right\| = \left\|\frac{1}{V(1/t_1) + V(1/t_2)}\right\|
= \frac{1}{{t_1}^{-1} + {t_2}^{-1}}, \quad t_1,t_2\in (0,\infty).
$$
Therefore,
$R$ preserves the norm of the harmonic mean of any pair
of strictly positive continuous functions.  It is plain that $R$ is not a generalized composition operator as in \eqref{eq:wco}.
\end{example}

The following example demonstrates that the closed subset $\mathrm{Y}_0$ appearing in Theorem \ref{thH01} can be proper.

\begin{example}\label{eg:HM-NCO}
Let $\mathrm{X}=[0,1]$ and $\mathrm{Y}=[0,1/2]\cup [5/4, 3/2]$.
Define $T: C(\mathrm{X})_{++}\to C(\mathrm{Y})_{++}$ by
\[
(Tf)(y)=\left\{\begin{array}{ll} f(2y),& 0\leq y\leq 1/2,\\
(2y-2)f(0), & 5/4\leq y\leq 3/2. \end{array}\right.
\]
While $T$ preserves the norm of the harmonic mean of any finite collection of strictly positive continuous
functions, plainly, it is not a composition operator.
\end{example}

\subsection{Tingley and other versions of the norm of mean preserver problem}\label{Ss:3.4}

Following \cite{MS15}, one  defines a  mean of nonnegative real numbers to be a sequence
of continuous functions $M_n: [0,\infty)^n \to [0,\infty)$ for $n\geq 2$, satisfying  the following conditions:
\begin{description}
  \item[Internality] $\min\{a_1,\ldots,a_n\} \leq M_n(a_1,\ldots,a_n) \leq \max\{a_1,\ldots,a_n\}$,
  \item[Monotonicity] $a_j\leq b_j$, $j=1,\ldots,n$, imply $M_n(a_1,\ldots,a_n)\leq M_n(b_1,\ldots,b_n)$, and
  \item[Homogeneity] $M_n(\lambda a_1,\ldots, \lambda a_n) = \lambda M_n(a_1,\ldots,a_n)$,
  \end{description}
  for all $a_1,\ldots,a_n$ and $ \lambda>0$.

Such a sequence  defines naturally a  mean $\mm$ of arbitrary finitely many  functions $f_1,f_2,\ldots, f_n$  in
$C_0(\mathrm{X})_+$ on a locally compact space $\mathrm{X}$ by
$$
(f_1\,\mm\, f_2\,\mm\, \cdots\,\mm\,  f_n)(x) = M_n(f_1(x),f_2(x), \ldots, f_n(x)), \quad x\in \mathrm{X}.
$$
Classical examples include the \textit{power means} $\mm^p$  associated to
$$
M^p_n(a_1,\ldots,a_n) = \left(\frac{1}{n}\sum_{j=1}^{n} a_j^p\right)^{1/p}, \quad p\in (-\infty,\infty).
$$
In particular, we have $\mm^{1}$ being the arithmetic mean, and $\mm^{-1}$ being the harmonic mean (when $p<0$, we usually assume that all numbers $a_j$ are positive).
Moreover, by convention, $\mm^0 := \lim_{p\to 0}\mm^p $ is the geometric mean.
For more examples and properties of general means, see, e.g., \cite{Bullen-Book03}.


With the success of dealing with the most important  means,
namely the arithmetic, geometric and harmonic means, it would be natural to consider the following problem.

\begin{problem}\label{prob:general-means}
  Let $\mm$ be a mean of positive continuous functions.
  Let $\mathrm{X},\mathrm{Y}$ be locally compact Hausdorff spaces.
Let $F(\mathrm{X})\subseteq C_0(\mathrm{X})_{+}$ and $F(\mathrm{Y})\subseteq C_{0}(\mathrm{Y})_{+}$ be subsets containing
sufficiently many functions to peak on  compact $G_\delta$ subsets of $\mathrm{X}$ and $\mathrm{Y}$, respectively.
Let  $T: F(\mathrm{X}) \rightarrow F(\mathrm{Y})$ be a map preserving the norm of the $\mm$ mean of any finite collection of elements in $F(\mathrm{X})$.

Can $T$ be extended to a generalized composition operator?
More precisely, can one find a locally compact subset $\mathrm{Y}_0$ of $\mathrm{Y}$ and a surjective continuous function $\tau: \mathrm{Y}_0\rightarrow \mathrm{X}$ such that
\begin{align*}
(Tf)(y)= f(\tau (y)), \quad y\in \mathrm{Y}_0, f\in F(\mathrm{X})?
\end{align*}
\end{problem}

Under some natural conditions,
Theorem 3 in  \cite{MS15} provides an affirmative answer to
Problem \ref{prob:general-means} when
$T:F(\mathrm{X})\to F(\mathrm{Y})$ is bijective,
and $\|Tf\,\mm\,  Tg\|=\|f\,\mm\,  g\|$ holds for any $f,g\in F(\mathrm{X})$.
Therefore,  Problem \ref{prob:general-means} is mainly for non-bijective norm of mean preservers.

For the special class of power means $\mm^p$,  Theorem \ref{thm:p-means} provides an affirmative answer to Problem \ref{prob:general-means}.
Below we present its proof.

\begin{proof}[Proof of Theorem \ref{thm:p-means}]
The cases when $p=1,0$ and $-1$ for
the arithmetic, geometric and harmonic means
are done in Theorems \ref{th301}, \ref{thG01} and \ref{thH01}, respectively.
When $p>0$, we consider the map $f\mapsto (Tf^{\frac 1p})^p$, which
preserves the norm of the arithmetic mean of any finite collection of elements.
When $p<0$, we consider the map $f\mapsto (Tf^{\,-\frac 1p})^{-p}$, which preserves the norm of the harmonic mean of any finitely many elements.
Applying Theorems \ref{th301} and \ref{thH01} respectively to these two maps, we arrive at the desired conclusions.
\end{proof}

In the case of general means $\mm$, Problem \ref{prob:general-means} seems to be really challenging.

Finally,
let us provide the proof of the positive answer to a
Tingley version (Corollary \ref{cor:3means-sphere}) of
 Problem \ref{pr101}.

\begin{proof}[Proof of Corollary \ref{cor:3means-sphere}]
Arguing as in the proof of Theorem \ref{thm:p-means},
it suffices to verify the assertions for  arithmetic, geometric and harmonic means.

We work for the cases of arithmetic and geometric means first.
Note that both $C_{0}(\mathrm{X})_{+,1}$ and $C_{0}(\mathrm{Y})_{+,1}$ contain
sufficiently many functions to peak on  compact $G_\delta$ sets in $\mathrm{X}$ and  $\mathrm{Y}$, respectively.
The assertions follows from Theorem \ref{thm:p-means}, or more precisely, Theorems \ref{th301} and \ref{thG01},
if the preserver transformation $T$ is also onto.
Therefore, it suffices to check that $T$ is onto $ C_{0}(\mathrm{Y})_{+,1}$ if $T$ has dense range.

Let $g\in  C_{0}(\mathrm{Y})_{+,1}$ and $f_n \in  C_{0}(\mathrm{X})_{+,1}$ be such that $Tf_n \to g$ in norm.
By the stated representation \eqref{eg:Tingley-rep}, we have
\begin{align*}
\| f_n - f_m\| &= \sup \{|f_n(x) - f_m(x)|: x\in \mathrm{X}\}\\
& = \sup \{|f_n(\tau(y)) - f_m(\tau(y))\| : y\in \mathrm{Y}_0\}\\
&= \sup \{ |(Tf_n)(y) - (Tf_m)(y)| : y\in \mathrm{Y}_0\}\\
& \leq \|Tf_n - Tf_m\| \to 0,\quad\text{as $m,n\to 0$}.
\end{align*}
Therefore, $\{f_n\}$ is a Cauchy sequence in $C_0(\mathrm{X})_{+,1}$ and thus converges to some $f\in C_0(\mathrm{X})_{+,1}$ in norm.
Clearly, the continuous functions $Tf$ and $g$ agree on $\mathrm{Y}_0$.
If  $\mathrm{Y}_0$ is dense in $\mathrm{Y}$, then $Tf=g$.  This says that $T$ is onto, as asserted.

So, we need to verify that $\mathrm{Y}_0$ is dense in $\mathrm{Y}$.
If it is not the case, we would have a nonzero function $g\in  C_{0}(\mathrm{Y})_{+,1}$ vanishing on $\mathrm{Y}_0$.
The above argument ensures that we would have a sequence $\{f_n\}$ in $C_{0}(\mathrm{X})_{+,1}$ converging to some $f\in C_{0}(\mathrm{X})_{+,1}$ in norm
such that $\{ Tf_n\}$ converges to  $g$ in norm.  Now $Tf$ agrees with $g$ on $\mathrm{Y}_0$, and thus
 $Tf$ vanishes on $\mathrm{Y}_0$.  This gives $f=0$, a contradiction since each $f_n$ has norm one. Therefore, $\mathrm{Y}_0$ is dense in $\mathrm{Y}$.

Now we consider the case when $X,Y$ are compact Hausdorff spaces and  $L: C(X)_{+,1}^{-1}\to C(Y)_{+,1}^{-1}$ preserves
the norm of the harmonic mean of arbitrary finite family of strictly positive functions of norm one.
We proceed as above with Theorem  \ref{thH01} instead, but
note that
we might need to verify that the limit $f$ of the Cauchy sequence $\{f_n\}$ of strictly positive functions of norm one
is again strictly positive.

Because $g$ is strictly positive,
$$
\lambda := \frac{1}{2}\min_{y\in Y} g(y) >0.
$$
Since $Tf_n\to g$, for large enough $n$ we have that
$$
f_n(\tau(y)) = (Tf_n)(y) > \lambda, \quad y\in Y_0,
$$
and since $\tau(Y_0)=X$, we see that
$$
\min_{x\in X} f_n(x) > \lambda,
$$
for all large enough $n$.  Thus the  limit $f$ of the sequence $\{f_n\}$ is strictly positive, as asserted.
The rest goes exactly the same way as above, and the proof is complete.
\end{proof}

\begin{remark}
  In Section \ref{S:3},
  all properties we make use of the domain $F(\mathrm{X})$ of a map $T$ which preserves the norm of a mean are those listed in Lemma \ref{lem:F(X)}.
  Therefore, we can restate  Theorems \ref{thm:p-means}, \ref{th301}, \ref{thG01} and \ref{thH01} and
  Corollary \ref{cor:3means-sphere}, by assuming that
  $F(\mathrm{X})$, as well as $F(\mathrm{Y})=T(F(\mathrm{X}))$, satisfies the three conditions in Lemma \ref{lem:F(X)} instead.  Indeed, these conditions
  merely say that the peak sets of functions in $F(\mathrm{X})$ determine the topology of $\mathrm{X}$.
\end{remark}

\section*{acknowledgements} Dong is supported in part
by the Natural Science Foundation of China (11671314). Li is the corresponding author.
Moln\'ar acknowledges supports from
the Ministry for Innovation and Technology, Hungary,
grant TUDFO/47138-1/2019-ITM and from the National Research, Development and Innovation Office of Hungary, NKFIH, Grant No. K115383, K134944.
Wong acknowledges support from the Taiwan MOST grant 108-2115-M-110-004-MY2, and
he expresses his deep appreciation to the colleagues in Tiangong University
and Nankai University
for their warm hospitality during his visit  there in 2019,
when this project started.


\end{document}